\documentclass[12pt,leqno]{amsart}
\usepackage{amssymb,amsthm}
\usepackage{amsmath}
\usepackage{setspace}
\usepackage{dcpic,pictexwd}
\usepackage[all]{xy}
\usepackage[pdftex]{graphicx}
\usepackage{mathrsfs}
\usepackage[pdftex]{hyperref}
\usepackage{bbm}
\usepackage{tikz}
\usetikzlibrary{cd}

\theoremstyle{definition}
 \newtheorem{definition}{Definition}[section]

\theoremstyle{plain}
 \newtheorem{proposition}[definition]{Proposition}

\theoremstyle{plain}
 \newtheorem{theorem}[definition]{Theorem}
 
\theoremstyle{definition}
 \newtheorem{example}[definition]{Example}

\theoremstyle{plain}
 \newtheorem{lemma}[definition]{Lemma}

\theoremstyle{plain}
 \newtheorem{corollary}[definition]{Corollary}

\theoremstyle{remark}

\theoremstyle{definition}

\theoremstyle{plain}

\newcommand{\Ext}{\mathrm{Ext}}

\newcommand{\Hom}{\mathrm{Hom}}

\newcommand{\Fun}{\mathrm{F}}
\renewcommand{\H}{\mathrm{H}}

\newcommand{\Ob}{\mathrm{Ob}}

\newcommand{\Z}{\mathbb{Z}}

\newcommand{\A}{\Lambda}
\newcommand{\G}{\mathrm{G}}

\renewcommand{\k}{\Bbbk}

\newcommand{\Tr}{\mathscr{T}}
\newcommand{\Sub}{\mathscr{S}}
\newcommand{\Ab}{\mathscr{A}}

\newcommand{\Cb}{\mathscr{C}}
\newcommand{\Db}{\mathscr{D}}
\newcommand{\Eb}{\mathscr{E}}

\newcommand{\R}{\mathbb{R}}

\newcommand{\Obj}{\mathscr{O}}

\newcommand{\Fb}{\mathscr{F}}
\newcommand{\Po}{\mathscr{P}}

\hypersetup{
    bookmarks=true,         
    unicode=false,          
    pdftoolbar=true,        
    pdfmenubar=true,        
    pdffitwindow=false,     
    pdfstartview={FitH},    
    pdfnewwindow=true,      
    colorlinks=true,       
    linkcolor=red,          
    citecolor=blue,        
    filecolor=magenta,      
    urlcolor=blue           
}

\title[Exact weights for triangulated categories]{Exact weights and path metrics for triangulated categories and the derived category of persistence modules} 
\thanks{}

\author[Bubenik]{Peter Bubenik}
\address{Department of Mathematics, University of Florida, Gainesville, FL,  United States of America}
\email{peter.bubenik@ufl.edu}

\author[V\'elez-Marulanda]{Jos\'e A.\ V\'elez-Marulanda}
\address{Department of Applied Mathematics \& Physics, Valdosta State University, Valdosta, GA,  United States of America}
\email{javelezmarulanda@valdosta.edu}
\address{Facultad de Matem\'aticas e Ingenier\'{\i}as, Fundaci\'on Universitaria Konrad Lorenz, Bogot\'a D.C.,  Colombia}
\email{josea.velezm@konradlorenz.edu.co}


\begin{document}
\maketitle

\begin{abstract}

We define exact weights on a pretriangulated category to be nonnegative functions on objects satisfying a subadditivity condition with respect to distinguished triangles.
Such weights induce a metric on objects in the category,
which we call a path metric. Our exact weights generalize the rank functions of J.\ Chuang and A.\ Lazarev for triangulated categories and are analogous to the exact weights for an exact category given by the first author and J.\ Scott and D.\ Stanley. We show that (co)homological functors from a triangulated category to an abelian category with an additive weight induce an exact weight on the triangulated category. We prove that triangle equivalences induce an isometry for the path metrics induced by cohomological functors.
In the perfectly generated or compactly generated case, we use Brown representability to express the exact weight on the triangulated category.
We give three characterizations of exactness for a weight on a pretriangulated category and show that they are equivalent.
We also define Wasserstein distances for triangulated categories.
Finally, we apply our work to derived categories of persistence modules and to representations of continuous quivers of type $\mathbb{A}$.
\end{abstract}

\renewcommand{\labelenumi}{\textup{(\roman{enumi})}}
\renewcommand{\labelenumii}{\textup{(\roman{enumi}.\alph{enumii})}}
\renewcommand{\labelenumiii}{\textup{(\roman{enumi}.\alph{enumii}.\arabic{enumiii})}}

\numberwithin{equation}{section}

\section{Introduction}\label{into}

Let $\Tr$ be a triangulated category with triangulated subcategory $\Sub$. 
In Verdier localization \cite{verdier}, \cite[Theorem 2.1.8]{neeman} one formally inverts the morphisms in $\Tr$ whose cone is an object in $\Sub$. 
Inspired by the theory of optimal transportation~\cite{kantorovich,villani},
we consider a quantitative version. Assign each object in $\Sub$ a weight and define the cost of a morphism in $\Tr$ whose cone is in $\Sub$ to the be the weight of that cone. 
Define the distance between objects in $\Tr$
to be the infimum of the cost of a path between the two objects, 
where a path is a finite composition of maps with cones in $\Sub$ and formal inverses of maps with cones in $\Sub$ and their cost is the sum of the costs of these maps.

In more detail,
we define exact weights for 
pretriangulated categories 
\cite[Def. 1.1.2]{neeman} as follows. 
A {\it weight} on 
a category
$\Cb$ with 
zero object $0$
is a function 
from the objects of $\Cb$ to $[0,\infty]$
such that $w(0)=0$, and if $X$ and $Y$ are
isomorphic objects 
then $w(X)=w(Y)$. 
If  
$\Cb$ is a pretriangulated category, then we say that $w$ is {\it exact}  provided that for all distinguished triangles $X\to Y\to Z\to \Sigma X$, 
$w(Y) \leq w(X)+w(Z)$, and 
for all objects $X$, $w(\Sigma X) = w(X)$. 
We give a more general definition in Definition~\ref{exactweight}.

If, in addition, $\Cb$ is a triangulated category, 
$w$ is finite and
for all objects $X$ and $Y$,  
$w(X\oplus Y) = w(X) + w(Y)$  then $w$ is  
a {\it rank function},
as defined by J. Chuang and A. Lazarev, \cite[Def. 2.1]{chuang}.  
In Example~\ref{ex:levels}, we use {\it levels} in a triangulated category, as defined by L.L.\ Avramov et al. \cite{avramov3}, to give examples of exact weights that are not rank functions.
Rank functions have also been studied by T.\ Conde et al. \cite{conde}, and as mentioned by Berkouk in \cite[\S 5]{berkouk2}, the r\^{o}le of these rank functions in persistence theory should be investigated. 

In \cite{bubenik1}, the first author together with J. Scott and D. Stanley introduced 
exact weights and path metrics for exact categories. This approach allowed them to define a distance for generalized persistence modules whose indexing category is a measure space. By using this distance, they were also able to define Wasserstein distances for additive categories with the property that every object is isomorphic to a direct sum of objects whose endomorphism ring is local. In particular, they obtained that this definition coincides with the usual definition of Wasserstein distances for $1$-parameter persistence modules.  

Exact weights on triangulated categories and exact weights on exact categories are connected as follows. In one direction, consider an exact category $\Cb$, with and exact weight $w$
on $\mathcal{D}(\Cb)$, the derived category of $\Cb$, which is a triangulated category. If $0\to X\to Y\to Z\to 0$ is an admissible exact sequence in $\Cb$, then by e.g. \cite[Lemma 4.1.12]{krause4}, we obtain an distinguished triangle $X\to Y\to Z\to X[1]$ in $\mathcal{D}(\Cb)$, where $X$, $Y$ and $Z$ are viewed as stalk complexes concentrated in degree zero.  Since $w$ is exact (in the sense of this article), it follows that $w(X)\leq w(Y) + w(Z)$, $w(Y) \leq w(X)+w(Z)$, and $w(Z) \leq w(X) + w(Y)$. This shows in particular that exact weights on the derived category $\mathcal{D}(\Cb)$ induce exact weights on $\Cb$ (in the sense of \cite{bubenik1}).

In the other direction,
let $\Ab$ be an abelian category with a finite weight $w$  
that is additive, which means that if $0\to X\to Y\to Z\to 0$ is 
a short exact
sequence in
$\Ab$, then $w(Y) = w(X) + w(Z)$.
Now consider a pretriangulated category $\Tr$ and
a homological functor $\Fun: \Tr
\to \Ab$ (as in \ref{cohfun}).
We use $w$ and $\Fun$
to define a weight $\underline{w}_{\Fun}$ on 
$\Tr$ (Definition~\ref{def:w_underline})
and show that it is
exact (in the sense of this article).  
Our first main result, Theorem \ref{thm2.12}, 
proves that in this setting, triangle equivalences induce an isometry for the corresponding path metrics.
In the case that $\Tr$ is a perfectly generated or compactly generated triangulated category, $\Ab = \mathrm{Ab}$, and $\Fun$ is a contravariant homological functor (i.e. a cohomological functor) we use Brown representability (Theorem \ref{brownthm}) to express the corresponding path metric (Theorem~\ref{brownexam}).

Our second main result, Theorem \ref{thm1.5},  gives three equivalent
characterizations of exact weights for pretriangulated categories, 
and is an analog of the corresponding result for exact weights of an exact category
\cite[Thm. 3.28]{bubenik1}.

Next, we consider derived categories of persistence modules. In the case that the indexing category of the persistence module is a measure space, we define an exact weight for the derived category (Definition~\ref{cohmuint}).
Our third main result, Theorem \ref{thm4.6}, gives lower and upper bounds on the path metric of either the homotopy or derived category of persistence modules, and provides an analogous result to that in \cite[Thm. 1.2]{bubenik1}.

We finally discuss the abelian category $\mathrm{rep}_\k(\mathbb{A}_{\R,S})$ of finitely generated pointwise finitely dimensional continuous representations over the real line (as introduced by K. Igusa et al. in \cite{igusa-rock-todorov1}) in order to prove our 
fourth main 
result (Theorem \ref{thm4.3}), which involves Wasserstein distances between objects in $\mathrm{rep}_\k(\mathbb{A}_{\R,S})$ and derived equivalences. In particular, this provides an alternative approach to study derived categories of zigzag persistence modules as discussed in \cite{hiraoka}.

\subsection*{Related work}

Our work is primarily motivated by the study of persistence modules.
In \cite{hiraoka}, Y. Hiraoka et al. studied the bounded derived category of zigzag persistence modules by considering the bounded derived category of finite dimensional representations of an $\mathbb{A}_n$-type quiver $\mathcal{D}^b(\mathbb{A}_n\textup{-rep})$. Moreover, they used derived equivalences induced by tilting modules (see \cite[\S 1.7]{happel4}) to compute interleaving and bottleneck distances between complexes whose terms are zigzag persistence modules and prove a corresponding algebraic stability theorem (see \cite[Thm. 4.13]{hiraoka}).  An important fact used to obtain some of the main results in \cite{hiraoka} is that the category of finite dimensional representations of an $\mathbb{A}_n$-type quiver is a hereditary category, and thus every complex in $\mathcal{D}^b(\mathbb{A}_n\textup{-rep})$ is quasi-isomorphic to its own cohomology. On the other hand N. Berkouk has discussed interleaving distances for derived categories of multi-parameter persistence modules in \cite{berkouk2} and together with G. Ginot, obtained in \cite{berkouk1} an isometry theorem for the bounded derived category of constructible sheaves on $\mathbb{R}$. 
These results motivate the study of distances of
derived categories of more general categories of persistence modules.

Independently of \cite{bubenik1},
B. Guinti et al. introduced  the amplitude of an abelian category \cite[Def. 2.1]{Giunti:2021a}, which they showed is equivalent to a noise system of M. Scolamiero et al. \cite[Def. 6.1]{Scolamiero:2017}.
These notions are closely related to exact weights (in the sense of \cite{bubenik1}) but slightly less general. We invite the reader to look at  \cite{bubenik1, Giunti:2021a, Scolamiero:2017} for examples and applications of exact weights to topological data analysis.

There are at least two other interesting approaches to defining distances on triangulated categories.
A.\ Neeman~\cite{neeman2} has studied Lawvere metrics on triangulated categories and their Cauchy completions.
Biran et al.  ~\cite{biran} 
have defined a family of metrics on triangulated categories, called a \emph{fragmentation metrics}.
In contrast to our approach, which uses weights on objects and zigzags of morphisms, theirs uses weights on distinguished triangles and iterated cone decompositions. It also makes fundamental use of the octahedral axiom, which our approach does not require.

We follow the terminology and notation in \cite{krause4} and \cite{neeman}.
To make this article accessible to a broad audience, 
we have included numerous definitions and properties concerning triangulated categories in the Appendix \ref{appendix}.

\section{Exact weights and examples}

In this section we review the definitions of weights and path metrics for categories and define exact weights for triangulated categories,  consider the ones that are induced by additive weights on abelian categories and cohomological functors and prove our first and second main results (Theorem \ref{thm2.12} and Theorem \ref{thm1.5}). 
We also 
define
Wasserstein distances for
triangulated categories, and provide some of their properties. 

\subsection{Weights and path metrics for categories}

Let $\mathscr{X}$ be a class. Following \cite[Def. 3.1]{bubenik1}, a {\it weight} $w$ on $\mathscr{X}$ is a a map that assigns $w(M)\in [0,\infty]$ to each $M\in \mathscr{X}$.  A {\it symmetric Lawvere metric} on $\mathscr{X}$ is a map $d: \mathscr{X}\times \mathscr{X}\to [0,\infty]$ that satisfies the following properties for all $M,N,P\in \mathscr{X}$: (i) $d(M,M) =0$; (ii) $d(M,N)= d(N,M)$; and (iii) $d(M,P)\leq d(M,N)+d(N,P)$. Let $\Cb$ be a category.  A  {\it metric} on $\Cb$ is a symmetric Lawvere metric $d$ 
on the class of objects of $\Cb$ with the additional property that if $M$ and $N$ are isomorphic objects in $\Cb$, then $d(M,N)=0$. Let $\mathscr{M}$ be a class of morphisms in $\Cb$ with a weight $w$, and let $X$ and $Y$ be objects in $\Cb$. A {\it $\mathscr{M}$-zigzag} $\gamma$ from $X$ to $Y$ is a sequence of morphisms 
\begin{equation}\label{zigzag0}
\gamma: X = X_0\xrightarrow{\gamma_1} X_1\xleftarrow{\gamma_2}  X_2\xrightarrow{\gamma_3}\cdots \xleftarrow{\gamma_n} X_n = Y, 
\end{equation}
such that for all $1\leq i\leq n$, $\gamma_i \in \mathscr{M}$. In this situation, the {\it cost} of $\gamma$ is defined as 
\begin{equation}
\mathrm{cost}_w(\gamma) = \sum_{i=1}^nw(\gamma_i). 
\end{equation}
The {\it path distance} $d_w(X,Y)$ between $X$ and $Y$ is defined by
\begin{equation}
d_w(X,Y) = \inf \{\mathrm{cost}_w(\gamma)\},
\end{equation}  
the infimum of the set of costs for which there exists an
$\mathscr{M}$-zigzag $\gamma$ from $X$ to $Y$
with that cost. 
It follows from \cite[Lemma 3.5]{bubenik1} that 
if $\mathscr{M}$ includes 
the isomorphisms in $\Cb$ and the weight of each isomorphism is $0$, then $d_w$ defines a Lawvere metric on the objects of $\Cb$. We call $d_w$ the {\it path metric} defined by $\mathscr{M}$ and $w$.  

From now on, if $\Cb$ is an additive category with a class of objects $\mathcal{C}$ with weight $w$ and which contains the zero-objects, then we assume that $w(0)=0$ and if $A$ and $B$ are isomorphic objects in $\mathcal{C}$ then $w(A) = w(B)$.

\subsection{Weights and path metrics for pretriangulated categories}\label{sec2.2}

Let $\Tr$ be a pretriangulated category with suspension $\Sigma$ (as in \ref{traingcat}), let $\mathcal{S}$ be a class of objects in $\Tr$ that contains the zero-object and is closed under isomorphisms, and let $w$ be a weight on $\mathcal{S}$. 

Let $X$ and $Y$ be objects in $\Tr$ and let $\alpha: X\to Y$ be a morphism in $\Tr$. By the axiom (Tr1) of triangulated categories (as in \ref{traingcat}), $\alpha$ fits in an distinguished triangle  
\begin{equation}\label{inducedtriag0}
\xymatrix@=30pt{
X\ar[r]^{\alpha}&Y\ar[r]^{\beta}&C_{\alpha}\ar[r]^{\gamma}&\Sigma X.
}
\end{equation}
Moreover, if $X\xrightarrow{\alpha}Y\to C\to \Sigma X$ is another distinguished triangle in $\Tr$ containing $\alpha$, then it follows by e.g. \cite[Lemma 3.1.5]{krause4} that $C_\alpha\cong C$ in $\Tr$. We say that $\alpha$ is in $\mathscr{M}(\mathcal{S})$ if and only if $C_\alpha$ is $\mathcal{S}$. In this way, we obtain a weight $w: \mathscr{M}(\mathcal{S})\to [0,\infty]$ defined as $w(\alpha)=w(C_\alpha)$. In particular, if  $\alpha: X\to Y$ is an isomorphism in $\Tr$, then by e.g. \cite[Cor. 1.2.6]{neeman}, $X\xrightarrow{\alpha} Y\to 0 \to \Sigma X$ is an distinguished triangle in $\Tr$ and thus $\mathscr{M}(\mathcal{S})$ contains all isomorphisms in $\Tr$  
and furthermore $w(\alpha)=0$.
Let $X$ and $Y$ be arbitrary objects in $\Tr$ and let $\gamma$ be a $\mathscr{M}(\mathcal{S})$-zigzag from $X$ to $Y$ as in (\ref{zigzag0}). Then 

\begin{equation}
\mathrm{cost}_w(\gamma) = \sum_{i=1}^nw(C_{\gamma_i}), 
\end{equation}
where for all $1\leq i\leq n$, $C_{\gamma_i}$ is as in (\ref{inducedtriag0}). Therefore we obtain a path metric $d_w(X,Y)$ on the objects of $\Tr$ defined by $\mathscr{M}(\mathcal{S})$ and $w$.

Following \cite[Def. 3.19]{bubenik1}, for all objects $X$ in $\Tr$, we  let $|d_w|(X) = d_w(X,0)$ and denote by $|d_w|_\mathcal{S}$ the restriction of $|d_w|$ to the class of objects $\mathcal{S}$. Observe that $|d_w|(0) = 0$, and if $X$ is isomorphic to $X'$ in $\Tr$, then $|d_w|(X) = |d_w|(X')$.  
Following \cite[Def. 3.21]{bubenik1}, we say that $w$ is {\it stable} provided that $|d_w|_\Obj = w$, and that  $w$ {\it lower bounds its path metric} if for all objects $X$ and $Y$ in $\mathcal{S}$, $|w(X)-w(Y)|\leq d_w(X,Y)$.

We next adapt \cite[Lemma 3.20]{bubenik1} to our situation.

\begin{lemma}\label{lemma1.1}
\begin{enumerate}
\item $|d_w|_\mathcal{S}\leq w$;
\item for all $X$ and $Y$ in $\Tr$, $d_w(X,Y)\leq |d_w|(X) + |d_w|(Y)$;
\item for all $X$ and $Y$ in $\mathcal{S}$, $d_w(X,Y)\leq w(X)+w(Y)$.
\end{enumerate}
\end{lemma}

\begin{proof}
Let $X$ be an arbitrary object in $\mathcal{S}$ and consider the distinguished triangle $0\to X\xrightarrow{\mathrm{id}_X}X\to 0$. Thus the cost of the $\mathscr{M}(\mathcal{S})$-zigzag $0\to X$ is $w(X)$, which implies $d_w(X,0)\leq w(X)$. This proves (i). Assume next that $X$ and $Y$ are objects in $\Tr$. Then $d_w(X,Y)\leq d_w(X,0)+d_w(0,Y)= |d_w|(X)+|d_w|(Y)$. This proves (ii). Finally, (iii) follows from combining (ii) and (i). 
\end{proof}

\subsection{Exact weights for pretriangulated categories}

Let $\Tr$ and $\Sigma$, $\mathcal{S}$ and $w$ be as before. The following definition extends \cite[Def. 3.9]{bubenik1} to triangulated categories.

\begin{definition} \label{exactweight}
The weight $w$ is {\it exact} if for all distinguished triangles $X\to Y\to Z\to \Sigma X$ in $\Tr$ whose terms are in $\mathcal{S}$, $w(X)\leq w(Y)+w(Z)$, $w(Y)\leq w(X)+w(Z)$ and $w(Z)\leq w(X)+w(Y)$.
\end{definition}

Assume that $\mathcal{S}$ is closed under direct sums and that $w$ is an exact weight on $\mathcal{S}$ (as in Definition \ref{exactweight}). If $X$ and $Y$ are objects in $\mathcal{S}$, then  $w(X\oplus Y) \leq w(X) + w(Y)$. Indeed,  consider the the distinguished triangles $X\to X\to 0\to \Sigma X$ and $0\to Y\to Y\to 0$ induced by the identity morphisms on $X$ and $Y$, respectively. By \cite[Prop  1.2.1 \& Remark 1.2.2]{neeman}, the direct sum of these distinguished triangles $X\to X\oplus Y\to Y\to \Sigma X$ is also an distinguished triangle. Then by the exactness of $w$  we get $w(X\oplus Y)\leq w(X) + w(Y)$. 

We say that the class of objects $\mathcal{S}$ in $\Tr$ is {\it closed under $\Sigma$} if for all $X\in \mathcal{S}$ and $i\in \Z$, we have $\Sigma^iX\in \mathcal{S}$.

\begin{lemma}\label{lemma1.2}
Assume that $\mathcal{S}$ is closed under $\Sigma$, and that $w$ is a weight on $\mathcal{S}$.
\begin{enumerate}
\item The weight $w$ is exact (in the sense of Definition \ref{exactweight}) if and only if the two following conditions are satisfied:
\begin{enumerate}
\item For all objects $X$ in $\mathcal{S}$, $w(X)=w(\Sigma X)$.
\item For all distinguished triangles $X\to Y\to Z\to \Sigma X$ in $\Tr$ whose terms are in $\mathcal{S}$, $w(Y)\leq w(X) +w(Z)$. 
\end{enumerate}
\item If $w$ is an exact weight, then for all objects $X$, $Y$ in $\Tr$ and $i\in \Z$, $d_w(X,Y) = d_w(\Sigma^iX,\Sigma^iY)$.
\end{enumerate}
\end{lemma}

\begin{proof}
(i). Observe that (i.b) follows trivially from the definition of an exact weight (as in Definition \ref{exactweight}). Let $X$ be an object in $\mathcal{S}$. By using the axioms (Tr0) and (Tr2) of pretriangulated categories (as in \ref{traingcat}), we obtain a distinguished triangle $X\to 0\to \Sigma X\to \Sigma X$. Since $w$ is exact, we get $w(X)\leq w(0)+w(\Sigma X)=w(\Sigma X)$ and $w(\Sigma X)\leq w(0)+w(X)=w(X)$. Thus $w(X) = w(\Sigma X)$, which proves (i.a). Conversely, assume that $X\to Y\to Z\to \Sigma X$ is an distinguished triangle in $\Tr$. We need to prove that $w(X)\leq w(Y) +w(Z)$ and $w(Z) \leq w(X)+w(Y)$. Indeed, by using the axiom (Tr2) of triangulated categories, we obtain an distinguished triangle $Y\to Z\to \Sigma X\to \Sigma X$. Thus by using (i.a) and (i.b), we get $w(Z)\leq w(Y)+w(\Sigma X)= w(Y)+w(X)$. Similarly, we obtain $w(X)\leq w(Y)+w(Z)$.

(ii). Let $X$, $Y$ be objects in $\Tr$, $i\in \Z$ and $\epsilon>0$ be fixed. Then there is $\mathscr{M}(\mathcal{S})$-zigzag $\gamma$ between $\Sigma^iX$ and $\Sigma^iY$ as in (\ref{zigzag0}) such that $\mathrm{cost}_w(\gamma) < d_w(\Sigma^iX, \Sigma^iY) + \epsilon$. Since $\Sigma$ is an exact functor on $\Tr$ (as in \ref{triangleeq}), it follows that $\gamma$ induces a $\mathscr{M}(\mathcal{S})$-zigzag $\Sigma^{-i} \gamma$ between $X$ and $Y$ such that 
\begin{equation*}
\mathrm{cost}_w(\Sigma^{-i}\gamma)=\sum_{k=1}^nw(\Sigma^{-i}C_{\gamma_k}).
\end{equation*}
Since $w$ is an exact weight, we obtain by (i) that for all $1\leq k\leq n$, $w(\Sigma^{-i}C_{\gamma_k})=w(C_{\gamma_k})$. Therefore
\begin{align*}
d_w(X,Y)\leq \mathrm{cost}_w(\Sigma^{-i}\gamma) = \mathrm{cost}_w(\gamma)< d_w(\Sigma^iX,\Sigma^iY)+\epsilon,
\end{align*}
which implies $d_w(X,Y)\leq d_w(\Sigma^iX,\Sigma^iY)$. Similarly we obtain $d_w(\Sigma^iX,\Sigma^iY) \leq d_w(X,Y)$. This proves (ii) and finishes the proof of Lemma \ref{lemma1.2}.  
\end{proof}

Assume that $\Tr$ is a triangulated category. 
If $\mathcal{S}= \Ob(\Tr)$, and $w$ is a finite exact weight on $\mathcal{S}$ such that $w(X\oplus Y) =w(X) + w(Y)$ for all objects $X$ and $Y$ in $\Tr$, then $w$ is called a {\it rank function} \cite[Def. 2.1]{chuang}.  In the following example, we use {\it levels} in triangulated categories (as introduced by  L. L. Avramov et al. in \cite[\S 2]{avramov3}) to give an example of an exact weight (as in Definition \ref{exactweight}) that is not a rank function. 

\begin{example} \label{ex:levels}
Let $\Sub$ be a subcategory of a triangulated category $\Tr$. We say that $\Sub$ is {\it full} if every morphism in $\Tr$ between objects in $\Sub$ is contained in $\Sub$, and we say that $\Sub$ is {\it strict} if it is closed under isomorphisms in $\Tr$. Let $\Sub$ be a subcategory of $\Tr$. We denote by $\mathrm{add}^{\Sigma}(\Sub)$ the intersection of all strict and full subcategories of $\Tr$ that are closed under finite direct sums and all suspensions and which also contain $\Sub$. On the other hand, we denote by $\mathrm{smd}(\Sub)$ the intersection of all subcategories of $\Tr$  that are closed under isomorphisms and direct summands, and which also contains $\mathcal{S}$. Assume next that $\Sub$ and $\Sub'$ are both strict and full triangulated subcategories of $\Tr$ (as in \ref{sub3}). We denote by $\Sub\star \Sub'$ the full and strict category of $\Tr$ whose objects are those $Y$ in $\Tr$ such that there is an distinguished triangle $X\to Y\to Z\to \Sigma X$ such that $X$ is an object in $\Sub$ and $Z$ is an object in $\Sub'$.  For all integers $n\geq 0$, we let $\Sub^{\star n}$ be the zero-subcategory if $n=0$, $\Sub$ itself if $n=1$, and the $n$-fold $\star$-product $\Sub\star\cdots\star \Sub$ of $\Sub$ with itself if $n\geq 2$. We denote by $\mathrm{thick}_\Tr(\Sub)$ the intersection of all thick subcategories of $\Tr$ (as in \ref{sub3}) that contain $\Sub$, and for all integers $n\geq 0$, let $\mathrm{thick}^n_\Tr(\Sub)=\mathrm{smd}(\mathrm{add}^{\Sigma}(\Sub)^{\star n})$. In particular, we obtain a filtration of subcategories of $\Tr$
\begin{equation*}
\{0\} = \mathrm{thick}^0_{\Tr}(\Sub)\subseteq \mathrm{thick}^1_{\Tr}(\Sub)\subseteq \cdots \subseteq \bigcup_{n\in \mathbb{N}}\mathrm{thick}^n_{\Tr}(\Sub)=\mathrm{thick}_\Tr(\Sub).
\end{equation*}
\noindent
For each object $X$ in $\Tr$ we define its {\it $\Sub$-level} as the number 
\begin{equation*}
\mathrm{level}^{\Sub}_\Tr(X) =\inf\{n\in \mathbb{N}:M\in \Ob(\mathrm{thick}^n_{\Tr}(\Sub))\}.
\end{equation*}

The following statements follow from \cite[Lemma 2.4]{avramov3}:
\begin{enumerate}
\item For all objects $X$ in $\Tr$ and all integers $i\in \Z$, $\mathrm{level}^\Sub_\Tr(\Sigma^iX)=\mathrm{level}^\Sub_\Tr(X)$.
\item If $X\to Y\to Z\to \Sigma X$ is an distinguished triangle in $\Tr$, then 
\begin{equation*}
\mathrm{level}^\Sub_\Tr(Y)\leq \mathrm{level}^\Sub_\Tr(X) + \mathrm{level}^\Sub_\Tr(Z).
\end{equation*}
\item For all objects $X$ and $X'$ in $\Tr$, 
\begin{equation}\label{level}
\mathrm{level}^\Sub_\Tr(X\oplus X')=\max\{\mathrm{level}^\Sub_\Tr(X),\mathrm{level}^\Sub_\Tr(X')\}\leq \mathrm{level}^\Sub_\Tr(X)  + \mathrm{level}^\Sub_\Tr(X').
\end{equation}
\end{enumerate}
It follows by Lemma \ref{lemma1.2} (i) that $w=\mathrm{level}^{\Sub}_\Tr$ defines an exact weight on the class of objects in $\Tr$ and which is not a rank function for the inequality in (\ref{level}) is strict provided that $X$ and $X'$ are both nonzero. 
\end{example}
\begin{example}\label{exam1.4}
Let $\Ab$ be an abelian category (as in \ref{abelian}) and let $\Obj$ be a class of objects in $\Ab$ that contains the zero-object. Let $\mathcal{C}(\Ab)$ be the abelian category of complexes with terms in $\Ab$, and $\mathcal{D}(\Ab)$ be the derived category of $\Ab$  (as in \ref{sub6}), which is a triangulated category by \cite[Cor. 10.4.3]{weibel},  and let $\Obj^\mathcal{D}$ be the class of objects $X^\bullet$ in $\mathcal{D}(\Ab)$ such that the terms of $X^\bullet$ are in $\Obj$. Note that $\Obj^\mathcal{D}$ is closed under the suspension functor $\Sigma = [1]$.  Assume that $w:\Obj^{\mathcal{D}}\to [0,\infty]$ is an exact weight on $\Obj^{\mathcal{D}}$ (in the sense of Definition \ref{exactweight}). Let $0\to X\to Y \to Z\to 0$ be a short exact sequence in $\Ab$ whose terms are in $\Obj$.  Then by Lemma \ref{remA1}, we obtain an distinguished triangle $\overline{X}\to \overline{Y} \to \overline{Z} \to \overline{X}[1]$ in $\mathcal{D}(\Ab)$, where $\overline{X}$, $\overline{Y}$ and $\overline{Z}$ are the complexes in $\mathcal{D}(\Ab)$ concentrated in degree zero corresponding to $X$, $Y$ and $Z$, and which are also in $\Obj^\mathcal{D}$. Since $w$ is exact, we have $w(\overline{X})\leq w(\overline{Y})+w(\overline{Z})$, $w(\overline{Y})\leq w(\overline{X})+w(\overline{Z})$, and $w(\overline{Z})\leq w(\overline{X})+w(\overline{Y})$. Thus the exact weight $w$ on $\Obj^{\mathcal{D}}$ induces an exact weight on $\Obj$ in the sense of \cite[Def. 3.9]{bubenik1}, which we also denote by $w$ by letting for all $X\in \Obj$, $w(X) = w(\overline{X})$.  Assume that $f: X\to Y$ be a morphism in $\Ab$ such that $\ker f$ and $\mathrm{coker}\, f$ are in $\Obj$. By viewing $f$ as a morphism of complexes concentrated in degree zero in $\mathcal{C}(\Ab)$, we get that $C_f$, the mapping cone of $f$ (as in \ref{sub6}), is the complex $0\to X\xrightarrow{f} Y\to 0$ (in degrees $-1$ and $0$), which is quasi-isomorphic (as in  \ref{remA2}) to the complex $0\to \ker f \to \mathrm{coker}\, f \to 0$ (in degrees $-1$ and $0$) with zero differential. This implies that $C_f \cong \overline{\ker f}[1]\oplus \overline{\mathrm{coker}\, f}$ in $\mathcal{D}(\Ab)$. Thus $w(C_f)\leq w(\overline{\ker f}[1]) + w(\overline{\mathrm{coker}\, f})= w(\overline{\ker f})+ w(\overline{\mathrm{coker}\, f}
)$. Let $\gamma$ be an $\mathscr{M}(\Obj)$-zigzag between $X$ and $Y$ as in \cite[Def. 3.7]{bubenik1}. Then by \cite[Lemma 4.1.11]{krause4}, we obtain a $\mathscr{M}(\Obj^\mathcal{D})$-zigzag from $\overline{X}$ to $\overline{Y}$
\begin{equation}\label{zigzag1}
\overline{\gamma}: \overline{X} = \overline{X}_0\xrightarrow{\overline{\gamma}_1} \overline{X}_1\xleftarrow{\overline{\gamma}_2}  \overline{X}_2\xrightarrow{\overline{\gamma}_3}\cdots \xleftarrow{\overline{\gamma}_n} \overline{X}_n = \overline{Y}.
\end{equation}
Moreover, by the arguments in the proof of \cite[Lemma 4.1.12]{krause4}, for all $1\leq k\leq n$, $C_{\overline{\gamma}_k}$ is quasi-isomorphic to the mapping cone of $\gamma_k$. Thus,  by the argument above, we have 
\begin{equation*}
\mathrm{cost}_w^{\mathcal{D}(\Ab)}(\overline{\gamma}) = \sum_{k=1}^n w(C_{\overline{\gamma}_k})\leq \sum_{k=1}^n w(\overline{\ker \gamma_k})+ w(\overline{\mathrm{coker}\, \gamma_k}) =\sum_{k=1}^n w(\gamma_k) = \mathrm{cost}_w^{\Ab}(\gamma),
\end{equation*} 
where $\mathrm{cost}_w^{\Ab}(\gamma)$ is the cost of $\gamma$ as a $\mathscr{M}(\Obj)$-zigzag in the sense of \cite[Def. 3.7]{bubenik1}. If we let 
$d_w^{\mathcal{D}(\Ab)}(\overline{X},\overline{Y})=\inf_{\overline{\gamma}} \mathrm{cost}_w^{\mathcal{D}(\Ab)}(\overline{\gamma})$ and $d_w^\Ab(X,Y)=\inf_\gamma \mathrm{cost}_w^{\Ab}(\gamma)$,  we get $d_w^{\mathcal{D}(\Ab)}(\overline{X},\overline{Y}) \leq d_w^\Ab(X,Y)$. In particular, if $w$ is a rank function, it follows that $d_w^{\mathcal{D}(\Ab)}(\overline{X},\overline{Y})=d_w^\Ab(X,Y)$.
\end{example}

\begin{example}
Assume that $\k$ is a field and let $\A$ be the cyclic Nakayama algebra with a simple left $\A$-module $S_1$ and whose corresponding  indecomposable projective left $\A$-module has composition series $\begin{matrix}S_1\\S_1\\S_1\end{matrix}$. Let $\Obj$ be the class of finitely generated left $\A$-modules. For all left $\A$-modules $X$ in $\Obj$, we define $w(X) = \dim_\k X$. 
It is clear that $w$ defines an exact weight on $\Obj$.  
On the other hand,  it is well-known that the category of finitely generated left $\A$-modules is a {\it Frobenius category} and thus the stable category of finitely generated left $\A$-modules, which is denoted by $\A\textup{-\underline{mod}}$, is a triangulated category with suspension functor $\Sigma X = \Omega ^{-1} X$ (see  \ref{sec8}). Observe that the objects in $\A\textup{-\underline{mod}}$ are also objects in $\Obj$. Moreover, we have an distinguished triangle $S_1\to S_1\to 0\to \begin{matrix}S_1\\S_1\end{matrix}$ in $\A\textup{-\underline{mod}}$ induced by the identity morphism on $S_1$. By the axiom (Tr2) of triangulated categories, we obtain an distinguished triangle $S_1\to 0\to \begin{matrix}S_1\\S_1\end{matrix}\to \begin{matrix}S_1\\S_1\end{matrix}$. Thus $w\left(\begin{matrix}S_1\\S_1\end{matrix}\right) = 2$ and $w(S_1)+w(0) =1$, which proves that $w$ is not an exact weight on the objects in $\A\textup{-\underline{mod}}$.
\end{example}

\subsection{Exact weights induced by cohomological functors}\label{w-coh}

Let $\Tr$ and $\Sigma$, $\mathcal{S}$ be as in the beginning of \S \ref{sec2.2}. In particular, we assume that $\Tr$ is a pretriangulated category.

\begin{definition} \label{def:w_underline}
Let $\Fun: \Tr \to \Ab$ (resp. $\Fun: \Tr^\textup{op}\to \Ab$) be a homological (resp. cohomological) functor (as in \ref{cohfun}), where $\Ab$ is an abelian category. For all objects $X$ and morphisms $f$ in  $\Tr$,  and for all $i\in \Z$, we assume $\Fun^i(X) = \Fun(\Sigma^iX)$ and $\Fun^i(f) = \Fun(\Sigma^if)$. Assume that $\Obj$ is a class of objects in $\Ab$ with a weight $w$ and denote by $\mathcal{L}_{\Tr, \Fun, \Obj,w}$ the class of objects $X$ in $\Tr$ such that $\Fun^i(X)\in \Obj$ and $w(\Fun^i(X))<\infty$ for all $i\in \Z$, and only finitely many of the objects $\Fun^i(X)$ are non-zero. We define a weight $\underline{w}_{\Fun}: \mathcal{L}_{\Tr, \Fun, \Obj,w}\to [0,\infty)$ by
\begin{equation}\label{weightfunct}
\underline{w}_{\Fun}(X)=\left| \sum_{i\in \Z}(-1)^iw(\Fun^i(X))\right |.
\end{equation}   
\end{definition}
Recall that $w$ is said to be {\it additive} if for all short exact sequences $0\to A\to B\to C\to 0$ in $\Ab$ with terms in $\Obj$, $w(B) = w(A) + w(C)$.

\begin{lemma}\label{lem4.1}
Assume that $w:\Obj\to [0,\infty)$ is an additive weight, and that $\Obj$ is closed under kernels and cokernels. Then for all bounded long exact sequences in $\Ab$ of the form 
\begin{equation}\label{longexactab}
\cdots\to A_1\xrightarrow{\alpha_1} A_2\xrightarrow{\alpha_2}A_3\to \cdots \to A_{\ell-1}\xrightarrow{\alpha_{\ell-1}}A_\ell\to \cdots,
\end{equation}
with $A_j\in \Obj$ for all $j\in \Z$, we have
\begin{equation*}
\sum_{j\in \Z} (-1)^jw(A_j)=0.
\end{equation*}

\begin{proof}
Let $j\in \Z$ be fixed. We have a short exact sequence
\begin{equation*}
0\to \ker \alpha_j\to A_j\to \mathrm{coker}\,\alpha_j\to 0.
\end{equation*}
Since $w$ is assumed to be additive, it follows that 
\begin{equation*}
w(A_j)= w(\ker \alpha_j)+w(\mathrm{coker}\,\alpha_j).
\end{equation*}
\noindent
On the other hand, by exactness of (\ref{longexactab}), $\mathrm{coker}\,\alpha_i= \ker\alpha_{i+1}$. Therefore
\begin{align*}
\sum_{j\in \Z}(-1)^jw(A_j)&= \sum_{j\in \Z}(-1)^j\left(w(\ker\alpha_j)+w(\mathrm{coker}\,\alpha_j)\right)\\
&=\sum_{j\in \Z}(-1)^j(w(\ker\alpha_j)+w(\ker \alpha_{j+1}))\\
&=0.
\end{align*}
\end{proof}
\end{lemma}

\begin{proposition}\label{prop1.7}
If $w:\Obj\to [0,\infty)$ is additive and $\Obj$ is closed under kernels and cokernels, then $\underline{w}_\Fun: \mathcal{L}_{\Tr, \Fun, \Obj,w}\to [0,\infty)$ is exact.  
\end{proposition}

\begin{proof}
Assume that $X\to Y\to Z\to \Sigma X$ is an distinguished triangle in $\Tr$ whose terms are in $\mathcal{L}_{\Tr, \Fun, \Obj,w}$, and assume without loss of generality that $\Fun:\Tr^\textup{op}\to \Ab$ is a cohomological functor.  Then by \ref{cohfun}, we obtain a long exact sequence in $\Ab$ and whose terms are in $\Obj$ given by 
\begin{equation*}\label{co-fun}
\cdots\to \Fun^{i+1}(X)\to \Fun^{i+1}(Y)\to \Fun^{i+1}(Z)\to \Fun^{i}(X)\to \cdots.
\end{equation*}
\noindent
Since $w$ is assumed to be additive, it follows by Lemma \ref{lem4.1} that 
\begin{equation*}
\sum_{i\in \Z}(-1)^iw(\Fun^i(X)) + \sum_{i\in \Z}(-1)^iw(\Fun^i(Z))-\sum_{i\in \Z}(-1)^iw(\Fun^i(Y))=0.
\end{equation*}
\noindent
After taking absolute values, we get $\underline{w}_{\Fun}(X)\leq \underline{w}_{\Fun}(Y)+\underline{w}_{\Fun}(Z)$, $\underline{w}_{\Fun}(Y)\leq \underline{w}_{\Fun}(X)+\underline{w}_{\Fun}(Z)$, and $\underline{w}_{\Fun}(Z)\leq \underline{w}_{\Fun}(X)+\underline{w}_{\Fun}(Y)$. 
\end{proof}

Assume that $\Tr'$ is another triangulated category and that $\mathscr{E}: \Tr'\to \Tr$ is triangle-equivalence (as in \ref{triangleeq}). Assume that $\G: \Tr'\to \Ab$ is a homological functor that is naturally equivalent to the composition $\Fun\circ \mathscr{E}: \Tr'\to \Ab$, which is also a homological functor for $\mathscr{E}$ takes distinguished triangles in $\Tr'$ into distinguished triangles in $\Tr$.  Let $\mathscr{O}$ be a class of objects in $\Ab$ that is closed under kernels and cokernels, and let $w: \mathscr{O}\to [0,\infty)$ be an additive weight on $\mathscr{O}$. Then the weight $\underline{w}_{\G}:\mathcal{L}_{\Tr', \G, \mathscr{O},w}\to [0,\infty)$ defined by 
\begin{equation*}
\underline{w}_{\G}(X') = \left |\sum_{i\in \Z}(-1)^iw\left(\G^iX'\right) \right |.
\end{equation*}
\noindent
is exact and satisfies that $\underline{w}_{\G}(X')=\underline{w}_{\Fun}(\mathscr{E} X')$ for all objects $X'$ in $\mathcal{L}_{\Tr', \G, \mathscr{O},w}$. 
The following isometry theorem is consequence of the above discussion and the definition of triangle-equivalences (as in \ref{triangleeq}).

\begin{theorem}\label{thm2.12}
Let $\Tr$ and $\Tr'$ be triangulated categories and $\mathscr{E}: \Tr'\to \Tr$ be a triangle-equivalence. Assume that $\Fun: \Tr\to \Ab$ and $\G: \Tr'\to \Ab$ are homological functors such that $\G$ is naturally equivalent to $\Fun \circ \mathscr{E}$, where $\Ab$ is an abelian category. Let $\Obj$ be a class of objects in $\Ab$ that is closed under kernels and cokernels and with an additive weight $w: \Obj\to [0,\infty)$. Then  
\begin{equation}
d_{\underline{w}_{\G}}(X',Y')= d_{\underline{w}_{\Fun}}(\mathscr{E}X',\mathscr{E}Y').
\end{equation}
\end{theorem}
Observe that we can use dual arguments to obtain a version of Theorem \ref{thm2.12} for the case when we consider cohomological functors instead of homological functors.

Let $\mathrm{Ab}$ denote the category of all abelian groups, and let $S'$ be a fixed object in $\Tr'$. 
Then the representable functor $\Hom_{\Tr}(-, \mathscr{E}S'): \Tr^\textup{op}\to \mathrm{Ab}$ is a cohomological functor by \cite[Lemma 3.1.2]{krause4}. Moreover, the functors $\Hom_{\Tr'}(-,S')$ and $\Hom_\Tr(-,\mathscr{E}S')\circ \mathscr{E}$ are naturally equivalent. Therefore, we obtain the following consequence of Theorem \ref{thm2.12}. 

\begin{corollary}\label{cor2.12}
Let $\mathscr{O}$ be a class of 
abelian groups
that is closed under kernels and cokernels with an additive weight  $w: \mathscr{O}\to [0,\infty)$.  Let $\mathscr{E}: \Tr'\to \Tr$ be a triangle-equivalence. Then for all objects $S'$ in $\Tr'$, we have
\begin{equation*}
d_{\underline{w}_{\Hom_{\Tr'}(-,S')}}(X',Y')= d_{\underline{w}_{\Hom_{\Tr}(-,\mathscr{E}S')}}(\mathscr{E}X',\mathscr{E}Y')
\end{equation*}
\end{corollary}

\begin{example}
Let $\Ab$ be an additive category.  
\begin{enumerate}
\item Assume that $\Ab$ is abelian and let $\Obj$ be a class objects in $\Ab$ that is closed under kernels and cokernels with an additive weight $w$, and let $\mathcal{K}(\Ab)$ denote the homotopy category (as in \ref{sub6}), which is also a triangulated category by e.g. \cite[Prop. 10.2.4]{weibel}. 
Then by \cite[Prop. 4.1.3]{krause4} the $0$-th cohomology group $\H^0(M^\bullet)$ of a complex $M^\bullet$ with terms in $\Ab$ defines a homology functor $\H^0: \mathcal{K}(\Ab) \to \Ab$ (as in \ref{cohfun}). Thus by Proposition \ref{prop1.7} we obtain an exact weight $\underline{w}_{\H^0}:\mathcal{L}_{\mathcal{K}(\Ab), \H^0, \Obj,w}\to [0,\infty)$ defined for all $M^\bullet$ in $\mathcal{L}_{\mathcal{K}(\Ab), \H^0, \Obj,w}$ as 
\begin{equation*}
\underline{w}_{\H^0}(M^\bullet) = \left |\sum_{i\in \Z}(-1)^iw(\H^i(M^\bullet))\right |.
\end{equation*}
\item Let $\mathcal{D}(\Ab)$ be the derived category of $\Ab$ (as in \ref{sub6}). Assume that $\mathcal{D}(\Ab)$ is a $\k$-linear category with $\k$ a field (as in \ref{klinear}). Let $S^\bullet$ be a fixed object in $\mathcal{D}(\Ab)$. Then we obtain a cohomological functor $\Hom_{\mathcal{D}(\Ab)}(-,S^\bullet):\mathcal{D}(\Ab)^\textup{op}\to \k\textup{-Mod}$, where $\k\textup{-Mod}$ denotes the category of $\k$-vector spaces. If we denote by $\Obj$ the objects of $\k$-Mod that are finite dimensional $\k$-vector spaces, then by Proposition \ref{prop1.7} we obtain an exact weight $\underline{w}_{\Hom_{\mathcal{D}(\Ab)}(-,S^\bullet)}: \mathcal{L}_{\mathcal{D}(\Ab), \Hom_{\mathcal{D}(\Ab)}(-,S^\bullet),\Obj,\dim_\k}\to [0,\infty)$ defined for all $M^\bullet$ in $\mathcal{L}_{\mathcal{D}(\Ab), \Hom_{\mathcal{D}(\Ab)}(-,S^\bullet),\Obj,\dim_\k}$ by 
\begin{equation*}
\underline{w}_{\Hom_{\mathcal{D}(\Ab)}(-,S^\bullet)}(M^\bullet) = \left |\sum_{i\in \Z}(-1)^i\dim_\k\Hom_{\mathcal{D}(\Ab)}(M^\bullet[-i],S^\bullet)\right |.
\end{equation*}

\end{enumerate}

\end{example}

Combining
Proposition \ref{prop1.7}, 
with the Brown Representability Theorem (Theorem \ref{brownthm}) 
we have the following.

\begin{theorem}\label{brownexam}
Let $\mathscr{O}$ be a class of abelian groups
that is closed under kernels and cokernels with an additive weight  $w: \mathscr{O}\to [0,\infty)$.
Let $\Tr$ be a perfectly generated (as in \ref{brown}) or more generally, compactly generated (as in \ref{compact}) triangulated category. 
Let
$\Fun: \Tr^\textup{op}\to \mathrm{Ab}$ be a cohomological functor that  takes direct sums in $\Tr$ to products in $\mathrm{Ab}$.
Then there exists an object $S$ in $\Tr$ such that for all objects $X$ and $Y$ in $\Tr$
\begin{equation*}\label{isometry}
d_{\underline{w}_{\Fun}}(X,Y) = 
d_{\underline{w}_S}(X,Y),
\end{equation*}
where
\begin{equation*}
\underline{w}_S(X) = \left |\sum_{i\in \Z}(-1)^iw\left(\Hom_{\Tr}(\Sigma^{-i}X,S)\right)\right |.
\end{equation*}
\end{theorem}

Examples of compactly generated triangulated categories  are  $\mathcal{D}(A\textup{-Mod})$ the derived category of modules over an associative ring  $A$ (see \cite[\S 5.8]{krause3}) as well as $\mathcal{D}(\mathrm{Qcoh}\,\mathbb{X})$ the derived category of quasi-coherent sheaves on a scheme $\mathbb{X}$ (see  \cite[pg. 317]{krause4}).

\subsection{Equivalent characterizations of exact weights}

Let $\Tr$ and $\Sigma$, $\mathcal{S}$ and $w$ be as in the beginning of \S \ref{sec2.2}. In particular, assume that $\Tr$ is a pretriangulated category.

The following result gives a version of \cite[Thm. 3.28]{bubenik1} for pretriangulated categories. 

\begin{theorem}\label{thm1.5}
Assume that $\mathcal{S}$ is closed under $\Sigma$ and that for all distinguished triangles $X\to Y\to Z\to\Sigma X$ in $\Tr$, if two of $\{X,Y,Z\}$ are in $\mathcal{S}$, then so is the third.  Then the following statements are equivalent:
\begin{enumerate}
\item $w$ is exact;
\item $w$ is stable, and $w(\Sigma X) =w(X)$ for all objects $X$ in $\mathcal{S}$;
\item $w$ lower bounds its path metric, and $w(\Sigma X)\leq w(X)$ for all objects $X$ in $\mathcal{S}$.
\end{enumerate}
\end{theorem}

\begin{proof}
Assume that $w$ is exact. Note that by Lemma \ref{lemma1.1} (i), in order to prove that $w$ is stable it is enough to prove that $|d_w|_\mathcal{S} \geq w$. Assume by contradiction that $|d_w|_\mathcal{S} < w$. Therefore, there exists an object $X$ in $\mathcal{S}$ and a $\mathscr{M}(\mathcal{S})$-zigzag $\gamma$ from $X$ to $0$ such that $\mathrm{cost}_w(\gamma)< w(X)$. 
Then let $X\in \mathcal{S}$ and let $\gamma$ be a $\mathscr{M}(\mathcal{S})$-zigzag from $X$ to $0$  of minimal cost with the property that $\mathrm{cost}_w(\gamma)< w(X)$, and without loss of generality, assume that $\gamma$ is of the form 
\begin{equation*}
\gamma: X = X_0\xrightarrow{\gamma_1} X_1\xleftarrow{\gamma_2}  X_2\xrightarrow{\gamma_3}\cdots \xleftarrow{\gamma_n} X_n = 0
\end{equation*}
\noindent
Then by the axiom (Tr1) of pretriangulated categories (as in \ref{traingcat}), the morphism $\gamma_1$ fits in an distinguished triangle 
\begin{equation}\label{triagproof}
X\xrightarrow{\gamma_1} X_1\to C_{\gamma_1}\to \Sigma X, 
\end{equation}
where $C_{\gamma_1}$ is in $\mathcal{S}$.  Note that $X_1$ is also an object in $\mathcal{S}$ by hypothesis. This implies that  
\begin{equation*}
\gamma': X_1\xleftarrow{\gamma_2}  X_2\xrightarrow{\gamma_3}\cdots \xleftarrow{\gamma_n} X_n = 0
\end{equation*}
is a $\mathscr{M}(\mathcal{S})$-zigzag from $X_1$ to $0$ with $\mathrm{cost}_w(\gamma') < \mathrm{cost}_w(\gamma)$. By the minimality of $\gamma$, we have that $\mathrm{cost}_w(\gamma')\geq w(X_1)$.  Since $w$ is exact by hypothesis, we obtain by using the distinguished triangle (\ref{triagproof}) that $w(X)\leq w(X_1)+w(C_{\gamma_1})\leq \mathrm{cost}_w(\gamma')+w(C_{\gamma_1}) = \mathrm{cost}_w(
\gamma)$, which is a contradiction. Thus $|d_w|_\mathcal{S} \geq w$. The second statement in (i) follows from Lemma \ref{lemma1.2} (i). This proves (i) $\Rightarrow$ (ii). Next assume that the statement in (ii) holds and let $X\xrightarrow{\alpha} Y\to Z\to \Sigma X$ be an distinguished triangle in $\Tr$ whose terms are in $\mathcal{S}$. Then by using the hypothesis we have 
\begin{equation*}
w(X)= d_w(X,0) \leq d_w(X,Y)+d_w(Y,0)= d_w(X,Y) + w(Y)
\end{equation*} 
Note also that $\alpha: X\to Y$ induces a $\mathscr{M}(\mathcal{S})$-zigzag between $X$ and $Y$ with $\mathrm{cost}_w(\alpha) = w(Z)$. Thus $d_w(X,Y)\leq w(Z)$, and therefore $w(X)\leq w(Z) +w(Y)$. Moreover, 
\begin{equation*}
w(Y)=d_w(Y,0)\leq d_w(Y,X)+d_w(X,0)=d_w(X,Y)+w(X)\leq w(Z)+w(X).
\end{equation*}
On the other hand, by using the axiom (Tr2) of triangulated categories (as in \ref{traingcat}), we obtain the distinguished triangle $Y\to Z\to \Sigma X\to \Sigma Y$ and thus by hypothesis and by using a similar argument as above, we have
\begin{align*}
w(Z)&=d_w(Z,0)\leq d_w(Z,Y)+d_w(Y,0)= d_w(Y,Z) + w(Y) \leq w(\Sigma X)+w(Y)\leq w(X)+w(Y).
\end{align*}
This proves the implication (ii) $\Rightarrow$ (i).  Next we prove the implication (ii) $\Rightarrow$ (iii). By hypothesis, we have that for $X$ and $Y$ in $\mathcal{S}$, $|w(X)-w(Y)|= |d_w(X,0)-d_w(Y,0)|\leq d_w(X,Y)$ and thus $w$ lower bounds its path metric. The second statement in (iii) follows trivially from the second statement in (ii). 
Finally, we prove (iii) $\Rightarrow$ (ii). Assume that $w$ lower bounds its path metric and let $X$ be a fixed object in $\mathcal{S}$. Then $w(X) = |w(X)-w(0)|\leq d_w(X,0)= |d_w|_\mathcal{S}(X)$. Thus by Lemma \ref{lemma1.1} (i), $|d_w|_\mathcal{S}(X)=w(X)$. On the other hand, by looking at the distinguished triangle $X\to 0 \to \Sigma X\to \Sigma X$, we further have $w(X)\leq d_w(X,0)\leq w(\Sigma X)$. This implies with the second hypothesis in (iii) that $w(X)= w(\Sigma X)$. 
\end{proof}

\subsection{Wasserstein distances for triangulated categories (cf. \cite[\S 5]{bubenik1})}

Assume that $\Tr$ is a triangulated category with arbitrary direct sums and with a metric $d$, and let $\Tr_\ell$ be the class of objects in  $\Tr$ that are direct sums of objects that have local endomorphism rings. 

\begin{definition}
For all objects $X$ and $Y$ in $\Tr_\ell$ and $1\leq p\leq \infty$ define
\begin{equation*}
W_p(d)(X,Y)=\inf \|\{d(X_a,Y_a)\}_{a\in A}\|_p,
\end{equation*}
where the infimum is taken over all isomorphisms $X\cong \bigoplus_{a\in A} X_a$ and $Y\cong \bigoplus_{a\in A} Y_a$ and each $X_a$ (resp. $Y_a$) is either zero or has a local endomorphism ring. 
\end{definition}

Assume that $\Tr_\ell$ has an exact weight $w$. 
If $d_w$ is the associated path metric to $w$ and $p\in [0,\infty]$, then it follows from Lemma \ref{lemma1.2} (ii) that for all objects $X$, $Y$ in $\Tr_\ell$ and $i\in \Z$,
\begin{equation*}\label{rem2.1}
W_p(d_w)(\Sigma^iX,\Sigma^iY)=W_p(d_w)(X,Y).
\end{equation*}
The following result can be obtained by adjusting the arguments in the proofs of \cite[Lemma 5.3, Prop. 5.4 \& Prop. 5.13]{bubenik1}. 

\begin{proposition}\label{prop3.10}
Let $w$ be an exact weight on $\Tr_\ell$, let $d = d_w$ be the induced path metric on $\Tr_\ell$, and let $p\in [1,\infty]$ be fixed. Then we have the following.
\begin{enumerate}
\item If $X\cong \bigoplus_{a\in A}X_a$ and $Y\cong \bigoplus_{b\in B}Y_b$, where each $X_a$ and $Y_b$ are indecomposable in $\Tr_\ell$. Then 
\begin{equation*}
W_p(d)(X,Y)=\inf_\varphi\left\|\left(\|(d(X_c,N_{\varphi(c)}))_{c\in C}\|_p,\|(d(X_c,0))_{c\in A- C}\|_p,\|(d(0,Y_b))_{b\in B- \varphi(C)}\|_p\right)\right\|_p,
\end{equation*} 
where the infimum is over all matchings from $A$ to $B$, i.e. injective maps $\phi: C\to B$ with $C\subseteq A$.
\item $W_p(d)$ is a $p$-subadditive metric on $\Tr_\ell$, i.e.
\begin{equation*}
W_p(d)\left(\bigoplus_{a\in A} X_a,\bigoplus_{a\in A} Y_a\right)\leq \|\{W_p(d)(X_a, Y_a)\}\|_p.
\end{equation*}
\end{enumerate}
\end{proposition}

\section{Application to persistence modules}

In this section, we consider the homotopy and derived categories of persistence modules indexed by a small category whose set of objects is equipped with a measure.
We
prove our third main result (Theorem \ref{thm4.6}) that provides an analog to \cite[Thm. 1.2]{bubenik1} for these triangulated categories.

\subsection{Exact weights for triangulated categories of persistence modules}

Let $\Po$ be a small category whose set of objects $P=\Ob(\Po)$ has a measure $\mu$. Assume that $\Ab$ is an abelian category, and $\Obj$ is a class of objects in $\Ab$ with an additive weight $w$, and which is closed under kernels and cokernels.  Let $\Ab^{\Po}$ be the abelian category of persistence modules, i.e., covariant functors $M: \Po\to \Ab$.  Let $\Tr$ be either $\mathcal{K}(\Ab^{\Po})$ or $\mathcal{D}(\Ab^{\Po})$. Let $\mathscr{F}=\{\Fun_p\}_{p\in P}$ be a set of homological functors $\Fun_p: \Tr\to \Ab$ (resp. cohomological functors $\Fun_p: \Tr^\textup{op}\to \Ab$) indexed by $P$. 

\begin{definition}
Let $M^\bullet$ be an object in $\Tr$ that is in $\mathcal{L}_p = \mathcal{L}_{\Tr, \Fun_p, \Obj,w}$ for all $p\in P$ (see Definition \ref{def:w_underline}).  
We define a weight $\underline{w}_{\mathscr{F}}(M^\bullet): P\to [0,\infty)$ induced by $w$ and $\mathscr{F}$ as 
\begin{equation}
\underline{w}_{\mathscr{F}}(M^\bullet)(p) = \underline{w}_{\Fun_p}(M^\bullet),
\end{equation}
where the right hand side is given by Definition~\ref{def:w_underline}.
\end{definition}

The following result follows directly from the definition of $\underline{w}_{\mathscr{F}}$ and the exactness of $\underline{w}_{\Fun_p}$ for all $p\in P$.

\begin{lemma}\label{lem5.1}
Assume that $w$ is an additive weight. Let  $M^\bullet\to M'^\bullet\to M''^\bullet\to M^\bullet[1]$ be an distinguished triangle in $\Tr$ such that for all $p\in P$, $M^\bullet, M'^\bullet$ and $M''^\bullet$ are in $\mathcal{L}_p$. If $p\in P$, then we have the following:
\begin{enumerate}
\item $\underline{w}_{\Fb}(M^\bullet)(p)\leq \underline{w}_{\Fb}(M'^\bullet)(p) + \underline{w}_{\Fb}(M''^\bullet)(p)$;
\item  $\underline{w}_{\Fb}(M'^\bullet)(p)\leq \underline{w}_{\Fb}(M^\bullet)(p) + \underline{w}_{\Fb}(M''^\bullet)(p)$; 
\item $\underline{w}_{\Fb}(M''^\bullet)(p)\leq \underline{w}_{\Fb}(M^\bullet)(p) + \underline{w}_{\Fb}(M'^\bullet)(p)$. 
\end{enumerate}
\end{lemma}

\begin{definition}\label{cohmuint}
We denote by $\mathcal{T}^{\Tr}_{\mathscr{O},\mu,\underline{w}_\Fb}$ the class of objects $M^\bullet$ in $\Tr$ that are in $\mathcal{L}_p$ for all $p\in P$ such that $\underline{w}_\Fb(M^\bullet)(p)$ is $\mu$-integrable.  For all objects $M^\bullet$ in $\mathcal{T}^{\Tr}_{\mathscr{O},\mu,\underline{w}_\Fb}$, we let
\begin{equation}
(\mu \circ \underline{w}_\Fb)(M^\bullet)  = \int_{p\in P}\underline{w}_\Fb(M^\bullet)(p)d\mu(p)= \int_{P}\underline{w}_\Fb(M^\bullet)d\mu.
\end{equation}  
\end{definition}

\begin{lemma}\label{lem4.3}
If $w$ is an additive weight on $\mathscr{O}$, then $\mu \circ \underline{w}_\Fb$ is an exact weight on $\mathcal{T}^{\Tr}_{\mathscr{O},\mu,\underline{w}_\Fb}$.  
\end{lemma}
\begin{proof}
It is straightforward to prove that if $M^\bullet$ is the zero object in $\mathcal{T}^{\Tr}_{\mathscr{O},\mu,\underline{w}_\Fb}$, then $(\mu\circ \underline{w}_\Fb)(M^\bullet)=0$. Moreover, it is also straightforward to prove that if $M^\bullet$ and $M'^\bullet$ are isomorphic objects in $\mathcal{T}^{\Tr}_{\mathscr{O},\mu,\underline{w}_\Fb}$, then $(\mu\circ \underline{w}_\Fb)(M^\bullet)= (\mu\circ \underline{w}_\Fb)(M'^\bullet)$.
Let $M^\bullet\to M'^\bullet\to M''^\bullet\to M^\bullet[1]$ be an distinguished triangle in $\Tr$ whose terms are in $\mathcal{T}^{\Tr}_{\mathscr{O},\mu,\underline{w}_\Fb}$. Then by Lemma \ref{lem5.1}, we obtain for all $p\in P$ that $\underline{w}_\Fb(M^\bullet)(p)\leq \underline{w}_\Fb(M'^\bullet)(p) + \underline{w}_\Fb(M''^\bullet)(p)$. This implies 
\begin{equation*}
\int_{p\in P}\underline{w}_\Fb(M^\bullet)(p)d\mu(p)\leq \int_{p\in P}\underline{w}_\Fb(M'^\bullet)(p)d\mu(p)+\int_{p\in P}\underline{w}_\Fb(M''^\bullet)(p)d\mu(p), 
\end{equation*}
which implies $(\mu\circ \underline{w}_\Fb)(M^\bullet)\leq (\mu\circ \underline{w}_\Fb)(M'^\bullet)+(\mu\circ \underline{w}_\Fb)(M''^\bullet)$. The other cases are proved similarly.  
\end{proof}

Let $M^\bullet$ and $N^\bullet$ be objects in $\Tr$, and let $\gamma^\bullet$ be a $\mathcal{M}(\mathcal{T}^{\Tr}_{\mathscr{O},\mu,\underline{w}_\Fb})$-zigzag from $M^\bullet$ to $N^\bullet$. Assume that $\gamma^\bullet$ is of the form
\begin{equation}
M^\bullet\xrightarrow{\gamma_1^\bullet} M_1^\bullet\xleftarrow{\gamma_2^\bullet}M_2^\bullet\xrightarrow{\gamma_3^\bullet}\cdots\xleftarrow{\gamma_n^\bullet} N^\bullet
\end{equation}
\noindent
Define a map $\mathrm{cost}_{\underline{w}_\Fb}(\gamma^\bullet): P\to [0,\infty)$ as $\mathrm{cost}_{\underline{w}_\Fb}(\gamma^\bullet)(p) = \mathrm{cost}_{\underline{w}_{\Fun_p}}(\gamma^\bullet)$ for all $p\in P$.

\begin{lemma}\label{lemma2.4}
Let $M^\bullet$ and $N^\bullet$ be objects in $\Tr$, and let $\gamma^\bullet$ be a $\mathcal{M}(\mathcal{T}^{\Tr}_{\mathscr{O},\mu,\underline{w}_\Fb})$-zigzag from $M^\bullet$ to $N^\bullet$. Then $\mathrm{cost}_{\mu\circ \underline{w}_{\Fb}}(\gamma^\bullet) = \mu(\mathrm{cost}_{\underline{w}_{\Fb}}(\gamma^\bullet))$.
\end{lemma}
\begin{proof}
Assume that $\gamma^\bullet$ is of the form $M^\bullet\xrightarrow{\gamma_1^\bullet} M_1^\bullet\xleftarrow{\gamma_2^\bullet}M_2^\bullet\xrightarrow{\gamma_3^\bullet}\cdots\xleftarrow{\gamma_n^\bullet} N^\bullet$. Then:
\begin{align*}
\mathrm{cost}_{\mu\circ \underline{w}_{\Fb}}(\gamma^\bullet)&=\sum_{j=1}^n(\mu\circ\underline{w}_{\Fb})(C_{\gamma_j^\bullet})\\
&= \sum_{j=1}^n\int_{p\in P} \underline{w}_\Fb(C_{\gamma_j^\bullet})(p)d\mu(p)\\
&= \int_{p\in  P} \sum_{j=1}^n\underline{w}_{\Fun_p}(C_{\gamma_j^\bullet})d\mu(p)\\
&= \int_{p\in P} \mathrm{cost}_{\underline{w}_{\Fun_p}}(\gamma^\bullet)d\mu(p)\\
&= \int_{p\in P} \mathrm{cost}_{\underline{w}_\Fb}(\gamma^\bullet)(p)d\mu(p)\\
&=\mu(\mathrm{cost}_{\underline{w}_\Fb}(\gamma^\bullet)). \qedhere
\end{align*}
\end{proof}

\begin{proposition}\label{prop3.7}
Let $M^\bullet$ and $N^\bullet$ be objects in $\mathcal{T}^{\Tr}_{\mathscr{O},\mu,\underline{w}_\Fb}$. Then
\begin{equation*}
d_{\mu\circ \underline{w}_\Fb}(M^\bullet, N^\bullet)\leq \int_{P}(\underline{w}_\Fb(M^\bullet)+\underline{w}_\Fb(N^\bullet))d\mu = \mu(\underline{w}_\Fb(M^\bullet)+\underline{w}_\Fb(N^\bullet)).
\end{equation*}
\end{proposition}

\begin{proof}
By Lemma \ref{lemma1.1} (iii), we obtain the following:
\begin{align*}
d_{\mu\circ \underline{w}_\Fb}(M^\bullet, N^\bullet)&\leq (\mu\circ \underline{w}_\Fb)(M^\bullet)+(\mu\circ \underline{w}_\Fb)(N^\bullet)\\
&=\int_{P}\underline{w}_\Fb(M^\bullet)d\mu+\int_{P}\underline{w}_\Fb(N^\bullet)d\mu\\
&=\int_{P}\underline{w}_\Fb(M^\bullet)+\underline{w}_\Fb(N^\bullet)d\mu\\
&=\mu(\underline{w}_\Fb(M^\bullet)+\underline{w}_\Fb(N^\bullet)).\qedhere
\end{align*}
\end{proof}




\begin{proposition}\label{prop3.8}
Assume that $w$ is an additive weight on $\Obj$. For all objects $M^\bullet$ and $N^\bullet$ in $\mathcal{T}^{\Tr}_{\mathscr{O},\mu,\underline{w}_\Fb}$, we have 
\begin{equation*}
\mu\left(\left |\underline{w}_\Fb(M^\bullet)-\underline{w}_\Fb(N^\bullet)\right |\right)=\int_{P}\left |\underline{w}_\Fb(M^\bullet)-\underline{w}_\Fb(N^\bullet)\right |d\mu\leq d_{\mu\circ \underline{w}_\Fb}(M^\bullet, N^\bullet).
\end{equation*}
\end{proposition}

\begin{proof}
Let $\gamma^\bullet$ be a $\mathcal{M}(\mathcal{T}^{\mathcal{K}}_{\mathscr{O},\mu,\underline{w}_\Fun})$-zigzag from $M^\bullet$ to $N^\bullet$, and assume that $\gamma^\bullet$ is of the form $M^\bullet\xrightarrow{\gamma_1^\bullet}M_1^\bullet\xleftarrow{\gamma_2^\bullet}M_2^\bullet\xrightarrow{\gamma_3^\bullet}\cdots\xleftarrow{\gamma_n^\bullet}N^\bullet$. Let $p\in P$ be fixed. Since $w$ is additive, it follows by Proposition \ref{prop1.7} that $\underline{w}_{\Fun_p}$ is exact, which together with Theorem \ref{thm1.5} gives:
\begin{align*}
\mathrm{cost}_{\underline{w}_{\Fb}}(\gamma^\bullet)(p) &= \mathrm{cost}_{\underline{w}_{\Fun_p}}(\gamma^\bullet)\\
&\geq d_{\underline{w}_{\Fun_p}}(M^\bullet, N^\bullet)\\
&\geq \left | \underline{w}_{\Fun_p}(M^\bullet) - \underline{w}_{\Fun_p}(N^\bullet)\right |\\
&=\left | \underline{w}_\Fb(M^\bullet)(p) - \underline{w}_\Fb(N^\bullet)(p)\right |. 
\end{align*}
\noindent
This together with Lemma \ref{lemma2.4} implies that 
\begin{align*}
\mathrm{cost}_{\mu\circ\underline{w}_\Fb}(\gamma^\bullet)\geq \mu\left(\left | \underline{w}_\Fb(M^\bullet) - \underline{w}_\Fb(N^\bullet)\right |\right)=\int_{p\in P}\left | \underline{w}_\Fb(M^\bullet)(p) - \underline{w}_\Fb(N^\bullet)(p)\right |d\mu(p), 
\end{align*}
\noindent
and thus, $d_{\mu\circ \underline{w}_\Fb}(M^\bullet, N^\bullet)\geq \mu\left(\left | \underline{w}_\Fb(M^\bullet) - \underline{w}_\Fb(N^\bullet)\right |\right)$. 
\end{proof}

Combining Proposition \ref{prop3.7} with Proposition \ref{prop3.8}, we obtain the following result which is an analog of \cite[Thm. 1.7]{bubenik1}. 
\begin{theorem}\label{thm4.6}
Assume that $w$ is an additive weight on $\Obj$. For all objects $M^\bullet$ and $N^\bullet$ in $\mathcal{T}^{\Tr}_{\mathscr{O},\mu,\underline{w}_\Fb}$, we have 
\begin{equation*}
\int_{P}\left |\underline{w}_\Fb(M^\bullet)-\underline{w}_\Fb(N^\bullet)\right |d\mu\leq d_{\mu\circ \underline{w}_\Fb}(M^\bullet, N^\bullet) \leq \int_{P}(\underline{w}_\Fb(M^\bullet)+\underline{w}_\Fb(N^\bullet))d\mu .
\end{equation*}
\end{theorem}

\subsection{Homotopy categories of persistence modules}
Let $\mathcal{K}(\Ab^{\Po})$ be the homotopy category of $\Ab^{\Po}$. For all $p\in P$, let $(-)_p: \Ab^\Po\to \Ab$ be defined as $(-)_p(M) = M(p)$ for all objects $M$, and $(-)_p(f)= f(p)$ for all morphisms $f$ in $\Ab^\Po$. Then  $(-)_p$ defines a covariant functor from $\Ab^\Po$ to $\Ab$. Let $M^\bullet$ be a fixed object in $\mathcal{C}(\Ab^{\Po})$, and assume that $M^\bullet$ is of the form

\begin{equation*}
M^\bullet: \cdots\to M^{i-1}\xrightarrow{d_M^{i-1}}M^i\xrightarrow{d_M^i}M^{i+1}\to \cdots.
\end{equation*}

Let $p\in P$ be fixed but arbitrary. We let $M^\bullet(p)$ be the complex in $\mathcal{C}(\Ab)$ 
\begin{equation*}
M^\bullet(p): \cdots\to M^{i-1}(p)\xrightarrow{d_M^{i-1}(p)}M^i(p)\xrightarrow{d_M^i(p)}M^{i+1}(p)\to \cdots.
\end{equation*}
\noindent
Observe that $M^\bullet [1](p)\cong M^\bullet(p)[1]$ in $\mathcal{C}(\Ab)$. If $f^\bullet: M^\bullet \to M'^\bullet$ is a morphism in $\mathcal{C}(\Ab^\Po)$, then we obtain a natural morphism $f^\bullet(p): M^\bullet(p)\to N^\bullet(p)$ defined as $(f^\bullet(p))^i= f^i(p)$ for all $i\in \Z$. Thus if  $f^\bullet, g^\bullet : M^\bullet\to N^\bullet$ are morphisms in $\mathcal{C}(\Ab^\Po)$ that are homotopic (as in \ref{sub6}), then $f^\bullet(p), g^\bullet(p): M^\bullet(p)\to N^\bullet(p)$ are morphisms in $\mathcal{C}(\Ab)$ that are also homotopic. Thus we obtain a well-defined functor $(-)_p: \mathcal{K}(\Ab^\Po)\to \mathcal{K}(\Ab)$.

Let $M^\bullet$ be a complex in $\mathcal{C}(\Ab^\Po)$ and let $p\in P$ be fixed but arbitrary. Then for all $i\in \Z$, we have an isomorphism in $\Ab$:
\begin{equation}\label{cohomp}
\H^i(M^\bullet)(p) \cong \H^i(M^\bullet (p)). 
\end{equation}

Indeed, let $i\in \Z$ be fixed. Then by definition, we have a short exact sequence in $\Ab^\Po$ given as follows:
\begin{equation*}
0\to \mathrm{im}\,\delta^{i-1}\to \ker\delta^i\to \H^i(M^\bullet)\to 0.
\end{equation*}
Thus we get a short exact sequence in $\Ab$:
\begin{equation*}
0\to( \mathrm{im}\,\delta^{i-1})(p)\to (\ker\delta^i)(p)\to \H^i(M^\bullet)(p)\to 0.
\end{equation*}
Since $(\mathrm{im}\,\delta^{i-1})(p)\cong \mathrm{im}(\delta(p))^{i-1}$ and $(\ker \delta^i) (p) \cong \ker(\delta^i(p))$ , it follows that there exists a natural morphism $h :\H^i(M^\bullet)(p)\to \H^i(M^\bullet(p))$ in $\Ab$, which is also an isomorphism by using the Snake Lemma for abelian categories. 

The following result is straightforward from the definition of mapping cones. 
\begin{lemma}
If $f^\bullet: M^\bullet \to N^\bullet$ is a morphism of complexes in $\mathcal{K}(\Ab)$ whose mapping cone is $C_{f^\bullet}$ (as in \ref{sub6}), then for all $p\in P$, the mapping cone of the induced morphism of complexes $f^\bullet(p): M^\bullet(p) \to N^\bullet(p)$ is isomorphic to $C_{f^\bullet}(p)$.
\end{lemma}

By the axiom (Tr1) of triangulated categories and the discussion above, we have that if $M^\bullet\xrightarrow{f^\bullet} N^\bullet\to C_{f^\bullet}\to M^\bullet[1]$ is an distinguished triangle in $\mathcal{K}(\Ab^\Po)$, then for all $p\in P$ we obtain an distinguished triangle in $\mathcal{K}(\Ab)$ given by $M^\bullet(p)\xrightarrow{f^\bullet(p)} N^\bullet(p) \to C_{f^\bullet}(p) \to M^\bullet(p)[1]$. Let $\H^0: \mathcal{K}(\Ab)\to \Ab$ be the $0$-th cohomology functor, which is a homological functor. If $p\in P$, then the composition of functors $\H^0_p=\H^0\circ (-)_p:\mathcal{K}(\Ab^\Po)\to \Ab$ is also a homological functor.  Let $\mathcal{L}_p = \mathcal{L}_{\mathcal{K}(\Ab^\Po), \H^0_p, \Obj,w}$ and let $M^\bullet$ be an object in $\mathcal{K}(\Ab^\Po)$ that is in $\mathcal{L}_p$ for all $p\in P$, and let $\mathscr{H}^0 = \{ \H^0_p\}_{p\in P}$.  Thus for all $p\in P$, 

\begin{equation}
\underline{w}_{\mathscr{H}^0}(M^\bullet)(p) = \left |\sum_{i\in \Z}(-1)^iw(\H^i(M^\bullet (p))\right | .
\end{equation}

Assume that $\Ab = \k\textup{-Mod}$, $\mathscr{O} = \Ob(\k\textup{-mod})$ and let $w=\dim_\k$. Let $M^\bullet$ be an object in $\mathcal{K}(\Ab^\Po)$ such that $M^\bullet$ is in  $\mathcal{L}_p = \mathcal{L}_{\mathcal{K}(\Ab^\Po), \H^0_p, \Obj,\dim_\k}$ for all $p\in P$. Let $\mathbf{hdim}\, M^\bullet = \underline{w}_{\mathscr{H}^0}(M^\bullet)$.  Thus by using (\ref{cohomp}) we have that for all $p\in P$  the following:
\begin{equation*}
(\mathbf{hdim}\, M^\bullet) (p) = \left |\sum_{i\in \Z}(-1)^i\dim_\k\H^i(M^\bullet(p))\right | =  \left |\sum_{i\in \Z}(-1)^i\mathbf{dim}\,\H^i(M^\bullet)(p)\right | , 
\end{equation*}
where for all $i\in \Z$, $\mathbf{dim}\,\H^i(M^\bullet)$ is the dimension vector of $\H^i(M^\bullet)$, i.e., for all $p\in P$, $\mathbf{dim}\,\H^i(M^\bullet)(p) =\dim_\k\H^i(M(p))$. Since $\mathbf{dim}$ defines an additive weight on $\Ab^\Po$, it follows by Lemma \ref{lem4.3} that $\mu \circ \mathbf{hdim}$ defines an exact weight on $\mathcal{T}^{\mathcal{K}(\Ab^\Po)}_{\mathscr{O},\mu,\mathbf{hdim}}$. Thus by Proposition \ref{prop3.7} and Theorem \ref{thm4.6} for all objects $M^\bullet$ and $N^\bullet$ in $\mathcal{T}^{\mathcal{K}(\Ab^\Po)}_{\mathscr{O},\mu,\mathbf{hdim}}$, we have
\begin{equation*}
\int_{P}\left | \mathbf{hdim}\, M^\bullet - \mathbf{hdim}\, N^\bullet \right | d\mu \leq d_{\mu\circ \mathbf{hdim}}(M^\bullet, N^\bullet)\leq \int_{P}\left( \mathbf{hdim}\, M^\bullet +\mathbf{hdim}\, N^\bullet \right ) d\mu. 
\end{equation*}
In particular if $M^\bullet=\overline{M}$ and $N^\bullet=\overline{N}$ are complexes concentrated in degree zero, whose non-zero terms in $\Ab^\Po$ are $M$ and $N$, respectively, we get 

\begin{equation*}
\int_{P}\left | \mathbf{dim}\, \overline{M} - \mathbf{dim}\, \overline{N} \right | d\mu \leq d_{\mu\circ \mathbf{hdim}}(\overline{M}, \overline{N})\leq \int_{P}\left( \mathbf{dim}\, \overline{M} +\mathbf{dim}\, \overline{N} \right ) d\mu,
\end{equation*}
which is a version of \cite[Eqn. (4.8)]{bubenik1}.

\subsection{Derived categories of persistence modules}
As before, assume that $\Ab = \k\textup{-Mod}$, $\mathscr{O} = \Ob(\k\textup{-mod})$ and let $w=\dim_\k$. Let $p\in P$. We define a functor $S_p:\Po\to \Ab$ as follows:
\begin{align*}
S_p(p') = \begin{cases} \k, &\text{ if $p'=p$}\\ 0, &\text{ otherwise}\end{cases},
\end{align*} 
with all trivial internal morphisms. We call $S_p$ the {\it simple persistence module} corresponding to $p\in P$. We denote by $\overline{S_p}$ the corresponding object in $\mathcal{D}(\Ab^\Po)$ concentrated in degree zero. Consider the cohomological functor $\Hom_{\mathcal{D}(\Ab^\Po)}(-,\overline{S_p}): \mathcal{D}(\Ab^\Po)^\textup{op}\to \k\textup{-Mod}$ and let $\mathcal{L}_p = \mathcal{L}_{\mathcal{D}(\Ab^\Po), \Hom_{\mathcal{D}(\Ab^\Po)}(-,\overline{S_p}), \Obj,\dim_\k}$. Assume that $M^\bullet$ is in $\mathcal{L}_p$ for all $p\in P$. We define $|\chi|(M^\bullet): P\to [0,\infty)$ as
\begin{equation*}
|\chi|(M^\bullet)(p)=\left |\sum_{i\in \Z}(-1)^i\dim_\k\Hom_{\mathcal{D}(\Ab^\Po)}(M^\bullet[-i], \overline{S_p})\right |.
\end{equation*}
\noindent
Thus by Proposition \ref{prop3.7} and Theorem \ref{thm4.6}, for all complexes $M^\bullet$ and $N^\bullet$ in $\mathcal{T}^{\mathcal{D}(\Ab^\Po)}_{\mathscr{O},\mu,|\chi|}$, 
\begin{equation*}
\int_{P}\left | |\chi|(M^\bullet) - |\chi|(N^\bullet) \right | d\mu \leq d_{\mu\circ |\chi|}(M^\bullet, N^\bullet)\leq \int_{P}\left( |\chi|(M^\bullet) +|\chi|(N^\bullet) \right) d\mu.
\end{equation*}

\section{Application to continuous quivers of type $\mathbb{A}$}

In this section we apply the previous results to the abelian category of point-wise finite dimensional representations of the real line introduced by K. Igusa et al. in \cite{igusa-rock-todorov1} that generalizes the representations of quivers of type $\mathbb{A}_n$, and prove our fourth main result (Theorem \ref{thm4.3}). 

Following e.g. \cite[\S 1]{igusa-rock-todorov1}, a {\it continuous quiver of type $\mathbb{A}$} is a triple $\mathbb{A}_{\R,S} = (\R, S, \preceq)$ that satisfies the following conditions:
\begin{enumerate}
\item 
\begin{enumerate}
\item $S$ is a discrete subset of $\R$, possibly empty with no accumulation points; 
\item the order on $S\cup \{\pm \infty\}$ is induced by the usual order of $\R$ and $-\infty < s < \infty$ for all $s\in S$;
\item the elements of $S\cup \{\pm \infty\}$ are indexed by a subset of $\Z\cup \{\pm \infty\}$ such that $s_n$ denotes the element of $S\cup \{\pm \infty\}$ with index $n$. The indexing must adhere to the following conditions:
\begin{enumerate}
\item there exists $s_0\in S\cup \{\pm \infty\}$;
\item if $m\leq n$ in $\Z\cup \{\pm \infty\}$ and $s_m, s_n$ are the corresponding elements in  $S\cup \{\pm \infty\}$, then for all $p \in \Z\cup \{\pm \infty\}$ such that $m\leq p\leq n$, there is $s_p\in S\cup \{\pm \infty\}$.
\end{enumerate}  
\end{enumerate}
\item We have a new partial order $\preceq$ on $\R$, which is called the {\it orientation} of $\mathbb{A}_{\R, S}$, and which is given as follows:
\begin{enumerate}
\item if $n$ is even and $x, y \in [s_n,s_{n+1}]$, then $x\preceq y$ if and only if $x\leq y$ and $s_n,s_{n+1}\in S\cup \{\pm \infty\}$;
\item if $n$ is odd and $x, y \in [s_n,s_{n+1}]$, then $x\preceq y$ if and only if $x\geq y$ and $s_n,s_{n+1}\in S\cup \{\pm \infty\}$;
\end{enumerate} 
\end{enumerate}
We refer the reader to \cite[Exam. 1.1.4]{igusa-rock-todorov1} to obtain examples of $\mathbb{A}_{\R,S}$ for different cases for $S$. In particular, by letting $S=\emptyset$, we have two possible cases for $\preceq$. If $s_0=-\infty$ and $s_1=+\infty$, we obtain that $\preceq$ coincides with the usual order of $\R$. In this situation, we let $\mathbb{A}_{\R,\emptyset} = \mathbb{A}_{\R}$. If $s_{-1}=-\infty$ and $s_0 = +\infty$, we obtain that $\preceq$ is the opposite order $\leq^\text{op}$ on $\R$. In this situation, we let $\mathbb{A}_{\R,\emptyset} = \mathbb{A}_{\R}^\text{op}$

Let $\mathbb{A}_{\R,S}$ be a fixed continuous quiver of type $\mathbb{A}$. A {\it $\k$-representation} of $\mathbb{A}_{\R,S}$ is a functor $M: \R \to \k\textup{-Mod}$ such that for all $x,y,z\in \R$ we have:
\begin{enumerate}
\item if $y\preceq x$ in $\mathbb{A}_{\R,S}$, then there exists a $\k$-linear map $M(x,y): M(x)\to M(y)$;
\item if $z\preceq y \preceq x$ in $\mathbb{A}_{\R,S}$, then $M(x,z) = M(y,z) \circ M(x,y)$.  
\end{enumerate}  
In this situation, we say that $M$ is {\it pointwise finite-dimensional (pwf) } provided that $\dim_\k M(x)<\infty$ for all $x\in \R$.  For all subsets $P$ of $\R$, we say that $V$ is {\it indexed by $P$} if for all $x\not\in P$, $V_x=0$. We denote by $\mathrm{Rep}^{\textup{pwf}}(\mathbb{A}_{\R,S})$ the category of pwf representations of $\mathbb{A}_{\R,S}$. For all intervals  $I$ in $\R$, we denote by $M_I$ the {\it interval indecomposable representation} of $\mathbb{A}_{\R, S}$, which is given as follows:
\begin{align*}
M_I(x) = \begin{cases} \k, &\text{ $x\in I$}, \\ 0. &\text{ otherwise,}\end{cases} && M_I(x,y) = \begin{cases} \mathrm{id}_\k, &\text{ $y\preceq x$ in $I$,}\\ 0, &\text{ otherwise.}\end{cases}
\end{align*}

It follows by \cite[Thm. 2.3.2 \& Thm. 2.4.13]{igusa-rock-todorov1} that for all intervals $I\subset \R$, $M_I$ is an indecomposable representation of $\mathbb{A}_{R,S}$ and that every indecomposable pwf representation of $\mathbb{A}_{R,S}$ is isomorphic to $M_I$ for some interval $I$ in $\R$. Moreover, every pwf representation of $\mathbb{A}_{\R, S}$ is isomorphic to a direct sum of interval indecomposable representations.  
It also follows from \cite[Remark 2.4.16 \& Lemma 3.1.4]{igusa-rock-todorov1} that $\mathrm{Rep}^{\textup{pwf}}(\mathbb{A}_{\R,S})$ is a hereditary category (as in \ref{hereditarycat}) that is closed under kernels and cokernels and in which any morphism has a coinciding image and coimage. Let $\mathcal{D}(\mathrm{Rep}^{\textup{pwf}}(\mathbb{A}_{\R,S}))$ be the derived category of $\mathrm{Rep}^{\textup{pwf}}(\mathbb{A}_{\R,S})$ and let $M^\bullet$ be an object in $\mathcal{D}(\mathrm{Rep}^{\textup{pwf}}(\mathbb{A}_{\R,S}))$. It follows by \ref{hereditarycat} that there is an isomorphism in $\mathcal{D}(\mathrm{Rep}^{\textup{pwf}}(\mathbb{A}_{\R,S}))$:

\begin{equation*}
M^\bullet \cong \bigoplus_{i\in \Z}\H^i(M^\bullet)[-i].
\end{equation*}
\noindent
Therefore, for all $x\in \R$, we have an isomorphism in $\mathcal{D}^b(\k\textup{-mod})$:

\begin{equation*}
M^\bullet(x) \cong \bigoplus_{i\in \Z}\H^i(M^\bullet)(x)[-i].
\end{equation*}
\noindent
Since for each $i\in \Z$, $\H^i(M^\bullet)$ is in $\mathrm{Rep}^{\textup{pwf}}(\mathbb{A}_{\R,S})$, it follows by the discussion above  that 

\begin{align*}
\H^i(M^\bullet)\cong \bigoplus_{J_i\in \mathcal{I}_i}M_{J_i},
\end{align*}
where $\mathcal{I}_i$ is a collection of intervals in $\R$, and for all $J_i\in \mathcal{I}_i$, $M_{J_i}$ is the corresponding interval representation. Thus
\begin{equation}\label{directco}
M^\bullet \cong \bigoplus_{i\in \Z, J_i\in \mathcal{I}_i}M_{J_i}[-i].
\end{equation}

Let $x\in \mathbb{R}$ be fixed and let $S_x$ be the corresponding simple representation in  $\mathrm{Rep}^{\textup{pwf}}(\mathbb{A}_{\R,S})$ i.e. 
\begin{align*}
S_x(y) = \begin{cases} \k, &\text{ if $y=x$}\\ 0, &\text{ otherwise}\end{cases},
\end{align*} 
with all the internal morphisms zero. As before, we denote by $\overline{S_x}$ the corresponding object in $\mathcal{D}(\mathrm{Rep}^{\textup{pwf}}(\mathbb{A}_{\R,S}))$ concentrated in degree zero. Let $\mathscr{X} = \{\Hom_{\mathcal{D}(\mathrm{Rep}^{\textup{pwf}}(\mathbb{A}_{\R,S}))}(-,\overline{S_x})\}_{x\in \mathbb{R}}$.

Let $N^\bullet \cong \bigoplus_{i\in \Z, J_i\in \mathcal{I}_i}N_{J_i}[-i]$ be another object in $\mathcal{D}(\mathrm{Rep}^{\textup{pwf}}(\mathbb{A}_{\R,S}))$. If $\mu$ is the Lebesgue measure on $\mathbb{R}$, then by using Proposition \ref{prop3.10} together with  (\ref{rem2.1}) and (\ref{directco}) we have for all $1\leq p\leq \infty$ that 

\begin{align*}
W_p(d_{\mu\circ \mathscr{X}})(M^\bullet, N^\bullet) &\leq \inf\|\{W_p(d_{\mu\circ \mathscr{X}})(M_{J_i}[-i], N_{J_i}[-i])\}_{i\in \Z, J_i\in \mathcal{I}_i}\|_p\\
&=\inf\|\{W_p(d_{\mu\circ \mathscr{X}})(\overline{M}_{J_i}, \overline{N}_{J_i})\}_{i\in \Z, J_i\in \mathcal{I}_i}\|_p.
\end{align*}

Following \cite[Def. 3.1.3]{igusa-rock-todorov1}, we denote by $\mathrm{rep}_\k(\mathbb{A}_{\R,S})$ the full subcategory of $\mathrm{Rep}^{\textup{pwf}}(\mathbb{A}_{\R,S})$, whose objects are representations $M$ that are finitely generated by indecomposable projective representations, which are listed in \cite[Rem. 2.4.16]{igusa-rock-todorov1}.  More precisely, $M$ is an object in $\mathrm{rep}_\k(\mathbb{A}_{\R,S})$ if there exist indecomposable projective objects $P_1,\ldots,P_k$ in $\mathrm{Rep}^{\textup{pwf}}(\mathbb{A}_{\R,S})$ and an epimorphism $\bigoplus_{i=1}^kP_i\to M$. As noted in \cite[Def. 1.18]{igusa-rock-todorov2}, the indecomposable projective objects in $\mathrm{Rep}^{\textup{pwf}}(\mathbb{A}_{\R,S})$ and $\mathrm{rep}_\k(\mathbb{A}_{\R,S})$ coincide. It follows by \cite[Thm. 3.0.1]{igusa-rock-todorov1}, $\mathrm{rep}_\k(\mathbb{A}_{\R,S})$ is a Krull-Schmidt category such that if $M_I$ and $M_J$ are interval modules in $\mathrm{rep}_\k(\mathbb{A}_{\R,S})$, then for $i=0,1$, $\dim_\k\Ext^i(M_I,M_J)\leq 1$, and $\Ext^i(M_I,M_J)=0$ for all $i\geq 2$. 
Following the notation in \cite{igusa-rock-todorov2}, we let $\mathcal{D}^b(\mathbb{A}_{\R,S}) = \mathcal{D}^b(\mathrm{rep}_\k(\mathbb{A}_{\R,S}))$ be the bounded derived category of $\mathrm{rep}_\k(\mathbb{A}_{\R,S})$. In particular, for all $M^\bullet$ and  $M'^\bullet$ in  $\mathcal{D}^b(\mathbb{A}_{\R,S})$, $\dim_\k\Hom_{\mathcal{D}^b(\mathbb{A}_{\R,S})}(M^\bullet, M'^\bullet)<\infty$.

Let $\mathbb{A}_{\R,S'}$ be another continuous quiver of type $\mathbb{A}$. It follows from \cite[Thm. 2.3.5]{igusa-rock-todorov2} that the bounded derived categories $\mathcal{D}^b(\mathbb{A}_{\R,S})$ and $\mathcal{D}^b(\mathbb{A}_{\R,S'})$ are equivalent as triangulated categories (see \ref{triangleeq}) if and only if one of the following conditions holds:
\begin{enumerate}
\item $S$ and $S'$ are both finite;
\item each of $S$ and $S'$ is bounded on exactly one side;
\item $S$ and $S'$ are unbounded on both sides.
\end{enumerate}  

The following result is a direct consequence of the above discussion together with Corollary \ref{cor2.12}, and as mentioned in \S \ref{into}, it provides an alternative approach to study derived categories of zigzag persistence modules as discussed in \cite{hiraoka}.

\begin{theorem}\label{thm4.3}
Assume that $\mathbb{A}_{\R,S}$ and $\mathbb{A}_{\R,S'}$ are continuous quivers if type $\mathbb{A}$ such that there exists a triangle-equivalence $\mathscr{E}: \mathcal{D}^b(\mathbb{A}_{\R,S}) \to \mathcal{D}^b(\mathbb{A}_{\R,S'})$. 

\begin{enumerate}
\item Let $w = \dim_\k$, and let $L^\bullet$ be a fixed object in $\mathcal{D}^b(\mathbb{A}_{\R,S})$. Then for all $M^\bullet$ and $N^\bullet$ in $\mathcal{D}^b(\mathbb{A}_{\R,S})$,
\begin{equation*}
d_{\underline{w}_{\Hom_{\mathcal{D}^b(\mathbb{A}_{\R,S})}(-,L^\bullet)}}(M^\bullet,N^\bullet)= d_{\underline{w}_{\Hom_{\mathcal{D}^b(\mathbb{A}_{\R,S'})}(-,\mathscr{E}L^\bullet)}}(\mathscr{E}M^\bullet,\mathscr{E}N^\bullet).
\end{equation*}

\item Let $\mathscr{X} = \{\Hom_{\mathcal{D}^b(\mathbb{A}_{\R,S})}(-,\overline{S_x})\}_{x\in \mathbb{R}}$ and $\mathscr{E}\mathscr{X} = \{\Hom_{\mathcal{D}^b(\mathbb{A}_{\R,S'})}(-,\mathscr{E}\overline{S_x})\}_{x\in \mathbb{R}}$, where $S_x$ is the simple representation corresponding to $x\in \R$ in  $\mathrm{rep}_\k(\mathbb{A}_{\R,S})$. If $\mu$ is the Lebesgue measure on $\R$, then for all objects $M^\bullet$ and $N^\bullet$ in $\mathcal{D}^b(\mathbb{A}_{\R,S})$ and for all $1\leq p\leq \infty$,
\begin{equation*}
W_p(d_{\mu \circ \mathscr{X}})(M^\bullet, N^\bullet) = W_p(d_{\mu \circ \mathscr{E}\mathscr{X}})(\mathscr{E}M^\bullet,\mathscr{E}N^\bullet).
\end{equation*}

 \end{enumerate}

\end{theorem}

\section{Acknowledgments}

Most of this research was performed during the second author's visit to the Department of Mathematics at the University of Florida during Fall 2022  
and he
would like to express his gratitude 
for their hospitality during his visit. 
The Office of Academic Affairs at Valdosta State University supported the second author during this visit with a paid academic leave for one semester. The second author was also supported during Fall 2023 by the Facultad de  Matem\'aticas e Ingenier\'{\i}as at the Fundaci\'on Universitaria Konrad Lorenz as a Docente Investigador in this institution. 
The first author was partially supported 
by the National Science Foundation (NSF) grant DMS-2324353
and
by the Southeast Center for Mathematics and Biology, an NSF-Simons Research Center for Mathematics of Complex Biological Systems, under NSF Grant No. DMS-1764406 and Simons Foundation Grant No. 594594.


\appendix

\section{Background on Categories}\label{appendix}
In this section, we review the required definitions and results in order to obtain our results. For basic notions from category theory, we refer the reader to e.g. \cite[App. A]{weibel}.

\subsection{Additive categories}  (\cite[\S 2.1]{krause4} and \cite[App. A]{weibel})

A category $\Cb$ is {\it additive} if it admits finite products, including the product indexed over the empty set, for each pair of objects $X, Y\in \Ob(\Cb)$, the set of morphisms $\Hom_\Cb(X,Y)$ is an abelian group, and the composition maps 
\begin{equation*}
\Hom_\Cb(Y,Z)\times \Hom_\Cb(X,Y)\to \Hom_\Cb(X,Z)\\
\end{equation*} 
that send a pair $(\psi,\phi)$ to the composite $\psi\circ \phi$ are biadditive. If $\Cb$ is an additive category, then finite product and direct sums exist and coincide \cite[Lemma 2.1.1]{krause4}. Let $\Cb$ be an additive category and $X, Y \in \Ob(\Cb)$ be arbitrary objects. We denote by $X\oplus Y$ the (co)product of $X$ and $Y$ in $\Cb$ and call it their direct sum. We say that $X$ is {\it indecomposable} if $X$ is not isomorphic to a direct sum of two non-zero objects in $\Cb$, i.e. if $X$ is  isomorphic to $W\oplus Z$ in $\Cb$, then either $W\cong 0$ or $Z\cong 0$. Recall that the {\it kernel} of a morphism $f: Y\to Z$ in $\Cb$ is a morphism $\iota: X\to Y$ such that $f\circ \iota = 0$ and for every morphism $\epsilon:X'\to Y$ such that $f\circ \epsilon=0$, there exists a unique morphism $\epsilon': X'\to X$ in $\Cb$ such that $\epsilon = \iota\circ \epsilon'$. In this situation we write $\iota = \ker f$ and consider it as a sub-object of $Y$. Dually, the {\it cokernel} of $g: Y\to Z$ in $\Cb$ is a morphism $\pi: Z\to W$ such that $\pi\circ g =0$ and for every morphism $\lambda: Z\to W'$ such that $g\circ \lambda = 0$, there exists a unique morphism $\lambda': W'\to W$ such that  $\lambda'=\pi\circ \lambda$. In this situation we write $\pi = \mathrm{coker}\,g$ and consider it as a sub-object of $W$. Let $\Db$ be another additive category. A functor $\Fun: \Cb\to \Db$ is {\it additive} if $\Fun$ preserves finite products. 

\subsection{Abelian categories }(\cite[App. A.4]{weibel}))\label{abelian}
Let $\Cb$ be an additive category. We say that $\Cb$ is an {\it abelian category} if the following conditions are satisfied: 
\begin{itemize}
\item[(AB1)] Every morphism in $\Cb$ has a kernel and cokernel.
\item[(AB2)] Every monomorphism in $\Cb$ is the kernel of its cokernel.
\item[(AB3)] Every epimorphism in $\Cb$ is the cokernel of its kernel.
\end{itemize}

\subsection{$\k$-linear categories} (\cite[App. A.1 \& A.2]{assem3})\label{klinear}
Let $\k$ be a field and let $\Cb$ be an additive category. We say that  $\Cb$ is a {\it $\k$-linear category}  if the morphism set  $\Hom_\Cb(X,Y)$ is a $\k$-vector space and the composition morphism $\Hom_\Cb(X,Y)\times \Hom_\Cb(Y,Z)\to \Hom_\Cb(X,Z)$ is $\k$-bilinear.  If $\Cb$ and $\Db$ are $\k$-linear categories, we say that an additive functor $\Fun: \Cb\to \Db$ is a $\k$-linear functor if $\Fun$ induces a morphism between $\k$-vector spaces $\Hom_\Cb(X,Y)\to \Hom_\Db(\Fun X,\Fun Y)$. 

\subsection{Krull-Schmidt categories} (\cite[\S 4]{krause2})\label{sub0}
Let $\Cb$ be an additive category. We say that $\Cb$ is  a {\it Krull-Schmidt category} if every object decomposes into a finite direct sum of objects having local endomorphism rings. Such decomposition is unique up to the order of the direct summands \cite[Thm. 4.2]{krause2}.

\subsection{Exact categories} (\cite[\S 2.1]{krause4})\label{sec9} Let $\Cb$ be an additive category. A sequence of morphisms in $\Cb$
\begin{equation}\label{exactseq}
\xi: 0\to X\xrightarrow{\alpha} Y\xrightarrow{\beta} Z\to 0
\end{equation} 
is said to be {\it exact} if  $\alpha$ is the kernel of $\beta$ and $\beta$ is the cokernel of $\alpha$. An {\it exact category} is a pair $(\Cb, \Eb)$ consisting of an additive category $\Cb$ and a class $\Eb$ of exact sequences $\xi$  in $\Cb$ as (\ref{exactseq}) which we call {\it admissible} (and say that $\alpha$ and $\beta$ are an {\it admissible monomorphism} and an {\it admissible epimorphism}, respectively) and which satisfy the following axioms:

\begin{itemize}
\item[(EX1)] For all objects $X\in \Ob(\Cb)$, the identity morphism $\mathrm{id}_X: X\to X$ is both an admissible monomorphism and an admissible epimorphism.
\item[(EX2)] The composition of two admissible monomorphisms is an admissible monomorphism, and the composition of two admissible epimorphisms is an admissible epimorphism.
\item[(EX3)] Each pair of morphisms $X'\xleftarrow{\phi} X\xrightarrow{\alpha}Y$ with $\alpha$ an admissible monomorphism can be completed to a pushout diagram
\begin{equation*}
\xymatrix@=30pt{
X\ar[r]^{\alpha}\ar[d]_{\phi}&Y\ar[d]\\
X'\ar[r]^{\alpha'}&Y'
}
\end{equation*}
such that $\alpha'$ is an admissible monomorphism. And each pair of morphisms $Y\xrightarrow{\beta}Z\xleftarrow{\psi}Z'$ with $\beta$ an admissible epimorphism can be completed to a pullback diagram
\begin{equation*}
\xymatrix@=30pt{
Y'\ar[r]^{\beta'}\ar[d]&Z'\ar[d]^{\psi}\\
Y\ar[r]^{\beta}&Z
}
\end{equation*}
such that $\beta'$ is an admissible epimorphism.  
\end{itemize}
 
In particular, in the situation of axiom (EX3), the morphism $\phi$ induces an isomorphism $\mathrm{coker}\, \alpha\cong \mathrm{coker}\, \alpha'$, while $\psi$ induces an isomorphism $\ker\beta'\cong \ker \beta$. Two exact sequences $\xi$ and $\xi'$ in $\Eb$ are {\it equivalent} if there is a commutative diagram in $\Cb$ of the form: 
\begin{equation*}
\xymatrix@=30pt{
\xi:&0\ar[r]&X\ar[r]^{\alpha}\ar[d]^{\mathrm{id}_X}&Y\ar[r]^\beta\ar[d]^{\phi}&Z\ar[r]\ar[d]^{\mathrm{id}_Y}&0\\
\xi':&0\ar[r]&X\ar[r]^{\alpha}&Y\ar[r]^\beta&Z\ar[r]&0
}
\end{equation*}
\noindent
It follows that $\phi$ is an isomorphism.

To simplify the notation, we use $\Cb$ instead of $(\Cb, \Eb)$ to denote an exact category.

\subsection{Frobenius categories} (\cite[\S 2.1]{happel})\label{sec5}
Let $\Cb$ be an exact category. We say that an object $P\in \Ob(\Cb)$ is {\it projective} (resp. $I\in \Ob(\Cb)$) if all admissible exact sequences $0\to X\to Y\to P\to 0$ (resp. $0\to I\to Y\to Z\to 0$) in $\Cb$ split, i.e. $Y\cong X\oplus P$ (resp. $Y\cong I\oplus Z$). We say that $\Cb$ has {\it enough projectives} (resp. {\it enough injectives}) if for all objects $X\in \Ob(\Cb)$ there exists an admissible epimorphism (resp. admissible monomorphism) $P\to X$ (resp. $X\to I$) with $P$ projective (resp. $I$ injective) in $\Cb$. We say that $\Cb$ is {\it Frobenius} if it has enough projective and enough injectives and the projective and injective objects coincide.  If $\Cb$ is a Frobenius category and $X\in \Ob(\Cb)$, then the {\it injective hull} of $X$ is an projective-injective object $I(X)$ together with a monomorphism $\iota_X: X\to I(X)$ with the property that if $f:I(X)\to Y$ is a morphism such that $f\circ \iota_X$ is a monomorphism, then $f$ is also a monomorphism.

\subsection{Pretriangulated and triangulated categories} (\cite[Def. 1.1.2 \& Def. 1.3.13]{neeman})\label{traingcat}
A {\it suspended category} is a pair $(\Tr, \Sigma)$ consisting of an additive category $\Tr$ and a self-equivalence $\Sigma: \Tr\to \Tr$ which is called a {\it suspension} or {\it shift}. A {\it triangle} in $(\Tr,\Sigma)$ is a sequence $(\alpha, \beta, \gamma)$ of morphisms 
\begin{equation}
\xymatrix@=30pt{
X\ar[r]^{\alpha}&Y\ar[r]^{\beta}&Z\ar[r]^{\gamma}&\Sigma X.
}
\end{equation}

A morphism between two triangles $(\alpha, \beta, \gamma)$ and $(\alpha', \beta', \gamma')$ in $\Tr$ is a triple $(\phi_1,\phi_2,\phi_3)$ of morphisms in $\Tr$ that make the following diagram in $\Tr$ commutative.

\begin{equation}\label{morphtriang}
\xymatrix@=30pt{
X\ar[d]_{\phi_1}\ar[r]^{\alpha}&Y\ar[d]_{\phi_2}\ar[r]^{\beta}&Z\ar[d]_{\phi_2}\ar[r]^{\gamma}&\Sigma X\ar[d]_{\Sigma\phi_1}\\
X'\ar[r]^{\alpha'}&Y'\ar[r]^{\beta'}&Z'\ar[r]^{\gamma'}&\Sigma X'\\
}
\end{equation}
A {\it pretriangulated category} is a triple $(\Tr,\Sigma, \Xi)$ consisting of a suspended category $(\Tr, \Sigma)$ and a class $\Xi$ of distinguished triangles, which we call {\it distinguished triangles} that satisfy the following axioms:

\begin{itemize}
\item[(Tr0)] A triangle isomorphic to an distinguished triangle is a distinguished triangle. For each $X\in \Ob(\Tr)$, the triangle $X\xrightarrow{\mathrm{id}_X} X\to 0\to \Sigma X$ is distinguished.
\item[(Tr1)]  Each morphism $\alpha: X\to Y$ in $\Tr$  fits into an distinguished triangle 
\begin{equation}\label{inducedtriag}
\xymatrix@=30pt{
X\ar[r]^{\alpha}&Y\ar[r]^{\beta}&C_{\alpha}\ar[r]^{\gamma}&\Sigma X.
}
\end{equation}
\item[(Tr2)] A triangle $(\alpha, \beta, \gamma)$ is exact if and only if $(\beta, \gamma, -\Sigma\alpha)$ is exact.
\item[(Tr3)] Given two distinguished triangles $(\alpha,\beta,\gamma)$ and $(\alpha',\beta',\gamma')$, each pair of morphisms $\phi_1: X\to X'$ and $\phi_2: Y\to Y'$ that satisfy $\phi_2\circ \alpha = \alpha'\circ \phi_1$ can be completed to a morphism of distinguished triangles as in (\ref{morphtriang}). 
\end{itemize}

We use $\Tr$ instead of $(\Tr, \Sigma, \Xi)$ to denote a pretriangulated category. In the situation of (Tr1), the object $C_\alpha$ is unique up to isomorphism. It follows from \cite[Cor. 1.2.6]{neeman} that a morphism $\alpha: X\to Y$ in $\Tr$ is an isomorphism if and only if $C_\alpha\cong 0$ in $\Tr$.

We say that a pretriangulated category $\Tr$ is {\it triangulated} if it further satisfies the following axiom:
\begin{itemize}
\item[(Tr4)] ({\it The Octahedral Axiom}) Given distinguished triangles $(\alpha_1,\alpha_2,\alpha_3)$, $(\beta_1,\beta_2,\beta_3)$, and $(\gamma_1,\gamma_2,\gamma_3)$ with $ \gamma_1 = \beta_1\circ \alpha_1$, there exists an distinguished triangle $(\delta_1,\delta_2,\delta_3)$ making the following diagram in $\Tr$ commutative.

\begin{equation*}
\xymatrix@=30pt{
X\ar[d]_{\mathrm{id}_X}\ar[r]^{\alpha_1}&Y\ar[d]_{\beta_1}\ar[r]^{\alpha_2}&U\ar[d]_{\delta_1}\ar[r]^{\alpha_3}&\Sigma X\ar[d]_{\mathrm{id}_{\Sigma X}}\\
X\ar[r]^{\gamma_1}&Z\ar[d]^{\beta_2}\ar[r]^{\gamma_2}&V\ar[d]^{\delta_2}\ar[r]^{\gamma_3}&\Sigma X\ar[d]^{\Sigma\alpha_1}\\
&W\ar[d]^{\beta_3}\ar[r]^{\mathrm{id}_W}&W\ar[d]^{\delta_3}\ar[r]^{\beta_3}&\Sigma Y\\
&\Sigma Y\ar[r]^{\Sigma\alpha_2}&\Sigma U
}
\end{equation*}
\end{itemize}

\subsection{(Thick) triangulated subcategories (\cite[\S 3.1]{krause4})}\label{sub3}

Let $\Tr$ be a triangulated category. A non-empty full subcategory $\Sub$ of $\Tr$ is a {\it triangulated subcategory} if the following conditions are satisfied:
\begin{itemize}
\item[(TS1)] For all objects $X$ in $\Sub$ and $i\in \Z$, $\Sigma^iX$ is in $\Sub$.
\item[(TS2)] For all distinguished triangles $X\to Y\to Z\to \Sigma X$ in $\Tr$, if two objects from $\{X,Y,Z\}$ belong to $\Sub$, then also does the third.
\end{itemize}
A triangulated subcategory $\Sub$ is {\it thick} if in addition the following condition holds:
\begin{itemize}
\item[(TS3)] Every direct summand of an object in $\Sub$ belongs to $\Sub$, i.e., a decomposition $X=X'\oplus X''$ for $X$ in $\Sub$ implies $X'$ in $\Sub$. 
\end{itemize}

\subsection{The stable category of a Frobenius category} (\cite[\S 3.3]{krause4})\label{sec8}
Let $\Fb$ be a Frobenius category (as in \ref{sec5}). The stable category of $\Fb$ denoted by $\underline{\Fb}$, is defined as follows. The objects of $\underline{\Fb}$ are the same as those of $\Fb$, and two morphisms $f,g: X\to Y$ in  $\Fb$ are identified in $\underline{\Fb}$ provided that $f-g$ factors through a projective-injective object of $\Fb$. In particular, the projective-injective objects are isomorphic to the zero-object in $\underline{\Fb}$.  Moreover, it follows from \cite[Lemma 1.1]{chenzhang} that  $X\cong Y$ in $\underline{\Fb}$ if and only if  there are projective-injective objects $I$ and $J$ in $\Fb$ such that $X\oplus I\cong Y\oplus J$ in $\Fb$. For all $X\in \Ob(\Fb)$, let $\Sigma X$ be the cokernel of an injective hull $I(X)$, which is unique up to isomorphism. It follows by \cite[pg. 13]{happel} that $\Sigma$ induces an equivalence $\Sigma: \underline{\Fb}\to \underline{\Fb}$ such that $\underline{\Fb}$ has the structure of a triangulated category. More precisely, if $0\to X\to Y\to Z\to 0$ is an admissible exact sequence in $\Fb$, we obtain an induced commutative diagram in $\Fb$:

\begin{equation}\label{inducedcomm}
\xymatrix@=30pt{
0\ar[r]&X\ar[d]_{\mathrm{id}_X}\ar[r]^{\alpha}&Y\ar[d]\ar[r]^{\beta}&Z\ar[d]_{\gamma}\ar[r]&0\\
0\ar[r]&X\ar[r]^{\iota_X}&I(X)\ar[r]&\Sigma X\ar[r]&0\\
}
\end{equation}
which in turn induces an distinguished triangle in $\underline{\Fb}$:  

\begin{equation}\label{triang}
\xymatrix@=30pt{
X\ar[r]^{\underline{\alpha}}&Y\ar[r]^{\underline{\beta}}&Z\ar[r]^{\underline{\gamma}}&\Sigma X.
}
\end{equation}
\noindent
Conversely, assume that $(\alpha',\beta',\gamma')$ is an distinguished triangle in $\underline{\Fb}$ given by 

\begin{equation}\label{triang'}
\xymatrix@=30pt{
X'\ar[r]^{\underline{\alpha}'}&Y'\ar[r]^{\underline{\beta}'}&Z'\ar[r]^{\underline{\gamma}'}&\Sigma X'.
}
\end{equation}
Then by \cite[Lemma 1.2]{chenzhang} there exists an distinguished triangle in $\underline{\Fb}$ as in (\ref{inducedtriag}) that is isomorphic to (\ref{triang'}) and which is induced by the admissible exact sequence in $\Fb$ given by
\begin{equation}\label{sexact}
0\to X\xrightarrow{\begin{pmatrix}\alpha\\\iota_X\end{pmatrix}} Y\oplus I(X)\to C_\alpha\to 0.
\end{equation}

\subsection{Homological functors}{(\cite[Def. 1.1.7]{neeman})}\label{cohfun} A covariant functor $\H: \Tr\to \Ab$ from a pretriangulated category $\Tr$ to an abelian category $\Ab$ is called a {\it homological functor} if for all distinguished triangles $X\to Y\to Z\to \Sigma X$, the sequence

\begin{equation*}
\H(X)\to \H(Y)\to \H(Z).
\end{equation*}
is exact in $\Ab$.  In particular, we obtain a long exact sequence in $\Ab$:
\begin{equation}\label{fun}
\cdots \to \H(\Sigma^{-1}Z)\to \H(X)\to \H(Y)\to \H(Z)\to \H(\Sigma X)\to\cdots
\end{equation}
\noindent
In this situation, for all $i\in \Z$, we set $\H^i(X)=\H(\Sigma^iX)$. A {\it cohomological functor} from $\Tr$ to $\Ab$ is a homological functor $\Fun: \Tr^\textup{op}\to \Ab$ that satisfies (\ref{fun}) with all the arrows reversed. 

\subsection{Brown representability}(\cite[\S 3.4]{krause4})\label{brown}
Let $\Tr$ be a triangulated category with arbitrary direct sums. A triangulated subcategory $\Sub$ of $\Tr$ is called {\it localizing} if it is closed under all direct sums. Given a class $\mathscr{X}$ of objects in $\Tr$, we denote by $\mathrm{Loc}(\mathscr{X})$ the smallest localizing subcategory of $\Tr$ that contains $\mathscr{X}$. A set of objects $\mathscr{X}$ of $\Tr$ is called {\it perfectly generating} if $\mathrm{Loc}(\mathscr{X}) = \Tr$ and the following holds. Given a countable family of morphisms $X_i\to Y_i$ in $\Tr$ such that the map $\Hom_\Tr(S,X_i)\to \Hom_\Tr(S,Y_i)$ is surjective for all $i$ and $S\in \mathscr{X}$, the induced map 
\begin{equation*}
\Hom_\Tr(S,\bigoplus_iX_i)\to \Hom_\Tr(S,\bigoplus_iY_i)
\end{equation*}    
is surjective. We say that $\Tr$ is {\it perfectly generated} if $\Tr$ admits a perfectly generating set. The following result is known as the {\it Brown representability theorem} for a perfectly generated triangulated category. 

\begin{theorem}{(\cite[Thm. 3.4.5]{krause4})}\label{brownthm}
Let $\Tr$ be a perfectly generated triangulated category.
Then a functor $\Fun: \Tr^\textup{op}\to \mathrm{Ab}$ is cohomological and sends all direct sums in $\Tr$ to products in $\mathrm{Ab}$ if and only if $\Fun$ is naturally equivalent to $\Hom_\Tr(-,X)$ for some object $X$ in $\Tr$.
 \end{theorem}
 
\subsection{Compact objects}(\cite[\S 3.4]{krause4})\label{compact}
Let $\Tr$ be a triangulated category with arbitrary direct sums. An object $X$ in $\Tr$ is {\it compact} if for any morphism $\phi: X\to \bigoplus_{i\in I} Y_i$ in $\Tr$, there is a finite set $J\subseteq I$ such that $\phi$ factors through $\bigoplus_{i\in J} Y_i$. Thus $X$ is compact in $\Tr$ if and only if the canonical map 
\begin{equation*}
\bigoplus_{i\in I}\Hom_\Tr(X,Y_i)\to \Hom_{\Tr}(X,\bigoplus_{i\in I}Y_i)
\end{equation*}
is bijective for all direct sums $\bigoplus_{i\in I}Y_i$ in $\Tr$. A set of objects $\mathscr{X}$ of compact objects in $\Tr$ is called {\it compactly generating} if $\Tr$ has no proper localizing subcategory containing $\mathscr{X}$. In this case, we say that $\Tr$ is {\it compactly generated}.   The following result is the Brown representability theorem for compactly generated triangulated categories.

\begin{theorem}{(\cite[Thm. 3.4.16]{krause4})}\label{brownthm2}
Let $\Tr$ be a compactly generated triangulated category and let $\mathrm{Ab}$ be the category of abelian groups. Then a functor $\Fun: \Tr^\textup{op}\to \mathrm{Ab}$ is cohomological and sends all direct sums in $\Tr$ to products in $\mathrm{Ab}$ if and only if $\Fun$ is naturally equivalent to $\Hom_\Tr(-,X)$ for some object $X$ in $\Tr$.
\end{theorem} 
 
\subsection{Triangle-equivalences}
\label{triangleeq} 
({\cite[pg. 74]{krause4}})
Let $\Tr$ and $\Tr'$ be triangulated categories with suspension functors $\Sigma$ and $\Sigma'$, respectively. An additive functor $\mathscr{E}: \Tr'\to \Tr$ is said to be {\it exact} if there exists an invertible natural transformation $\alpha: \mathscr{E}\circ\Sigma' \to \Sigma\circ\mathscr{E}$ such that for all distinguished triangles $X'\to Y'\to Z'\to \Sigma' X'$ in $\Tr'$, we have that $\mathscr{E}X'\to \mathscr{E}Y'\to \mathscr{E}Z'\to \mathscr{E}\Sigma' X$ is an distinguished triangle in $\Tr$. If $\mathscr{E}$ is an equivalence of categories which is also exact, then we said that $\mathscr{E}$ is a {\it triangle-equivalence} between $\Tr$ and $\Tr'$ and say that $\Tr$ and $\Tr'$ are {\it triangle-equivalent}.  

\subsection{Categories of complexes, homotopy categories and derived categories (\cite[Chap. 4]{krause4})}\label{sub6}
Let $\Cb$ be an additive category.  We denote by $\mathcal{C}(\Cb)$ the {\it category of complexes} with terms in $\Cb$ which we describe as follows. The objects of $\mathcal{C}(\Cb)$ are of the form
\begin{equation*}
X^\bullet: \cdots\to X^{i-1}\xrightarrow{d_X^{i-1}}X^i\xrightarrow{d_X^i}X^{i+1}\to \cdots
\end{equation*} 
with $X^i \in \Ob(\Cb)$ and for each $i\in \Z$, $d_X^i: X^i\to X^{i+1}$ is a morphism in $\Cb$ such that $d_X^{i}\circ d_X^{i-1}=0$.   
A morphism $f^\bullet: X^\bullet \to Y^\bullet$ in $\mathcal{C}(\Cb)$ is given by a sequence $f^\bullet = \{f^i\}_{i\in \Z}$ of morphisms in $\Cb$ with $f^i: X^i\to Y^i$ such that $f^{i+1}\circ d_X^i=d_Y^i\circ f^i$ for all $i\in \Z$.  The {\it mapping cone} of $f^\bullet: X^\bullet\to Y^\bullet$ is the complex $C_{f^\bullet}$ in $\mathcal{C}(\Cb)$ defined as follows. For each $i\in \Z$, 

\begin{align*}
(C_{f^\bullet})^i= Y^i\oplus X^{i+1} \text{ and } d_{C_{f^\bullet}}^i=\begin{pmatrix}  d_Y^i&f^{i+1}\\0&-d_X^{i+1}\end{pmatrix}. 
\end{align*}
We say that $f^\bullet:X^\bullet \to Y^\bullet$ is {\it null-homotopic} if there are morphisms $\rho^i: X^i\to Y^{i-1}$ such that $f^i=d_Y^{i-1}\circ \rho^i+\rho^{i+1}\circ d^i_X$. We denote by $\mathcal{K}(\Cb)$ the {\it homotopy category} of $\Cb$ is the quotient of $\mathcal{C}(\Cb)$ by the ideal of all null-homotopic morphisms in $\mathcal{C}(\Cb)$. More precisely, the objects of $\mathcal{K}(\Cb)$ are the same as those of $\mathcal{C}(\Cb)$ and two morphisms $f^\bullet, g^\bullet: X^\bullet \to Y^\bullet$ in $\mathcal{C}(\Cb)$ are identified in $\mathcal{K}(\Cb)$ if and only if $f^\bullet - g^\bullet$ is null-homotopic. In this situation, we say that $f^\bullet$ and $g^\bullet$ are {\it homotopic}.  Let $X^\bullet\in \Ob(\mathcal{C}(\Cb))$ be fixed but arbitrary. If $\ell\in \Z$, then the {\it $\ell$-th shifting}  of $X^\bullet$ is the complex $X^\bullet[\ell]$ defined as $(X^\bullet[\ell])^i=X^{i+\ell}$ for all $i\in \Z$ and all the differentials multiplied by $(-1)^\ell$.  If $f^\bullet: X^\bullet \to Y^\bullet$ is a morphism in $\mathcal{C}(\Cb)$, then we have a natural morphism $f^\bullet[\ell]: X^\bullet[\ell]\to Y^\bullet[\ell]$ such that $(f^\bullet[\ell])^i = f^{i+\ell}$ for all $i\in \Z$.
\noindent

The above definitions are compatible with homotopies and thus for $\ell\in \Z$ we also have a well-defined induced functor $[\ell]$ on $\mathcal{K}(\Cb)$. Moreover, it follows by e.g. \cite[Prop. 10.2.4]{weibel} that $\mathcal{K}(\Cb)$ is a triangulated category with suspension functor $\Sigma = [1]$ and where the distinguished triangles are of the form 
\begin{equation*}
X^\bullet \xrightarrow{h^\bullet} Y^\bullet \to C_{h^\bullet} \to X^\bullet[1].
\end{equation*} 

We say that a complex $X^\bullet$ in $\mathcal{C}(\Cb)$ is {\it acyclic} if for each $i\in \Z$, there exists an admissible sequence in $\Cb$
\begin{equation*}
0\to Z^i\xrightarrow{\alpha^i}X^i\xrightarrow{\beta^i}Z^{i+1}\to 0
\end{equation*} 
such that $d_X^i=\alpha^{i+1}\circ \beta^i$. We denote by $\mathbf{Ac}(\Cb)$ the full subcategory of complexes $\mathcal{C}(\Cb)$ that are isomorphic to an acyclic complex in $\mathcal{K}(\Cb)$. We say that $h^\bullet: X^\bullet \to Y^\bullet$ in $\mathcal{C}(\Cb)$ is a {\it quasi-isomorphism} if its mapping cone $C_{h^\bullet}$ is in $\mathbf{Ac}(\Cb)$. The {\it derived category} $\mathcal{D}(\Cb)$ of $\Cb$ is the localization of $\mathcal{K}(\Cb)$ by the the subcategory $\mathbf{Ac}(\Cb)$. More precisely, the objects of $\mathcal{D}(\Cb)$ are the same as those in $\mathcal{K}(\Cb)$ (and therefore the same as those in $\mathcal{C}(\Cb)$) and a morphism $f^\bullet: X^\bullet \to Y^\bullet$ in $\mathcal{D}(\Cb)$ is of the form $X^\bullet\xleftarrow{q^\bullet} Z^\bullet\xrightarrow{h^\bullet} Y^\bullet$, where $q^\bullet$ is a quasi-isomorphism in $\mathcal{C}(\Cb)$ and $h^\bullet$ is in $\mathcal{K}(\Cb)$. Two morphisms $f^\bullet, g^\bullet: X^\bullet \to Y^\bullet$ in $\mathcal{D}(\Cb)$ of the form  $X^\bullet\xleftarrow{q^\bullet} Z^\bullet\xrightarrow{h^\bullet} Y^\bullet$ and $X^\bullet\xleftarrow{q'^\bullet} Z'^\bullet\xrightarrow{h'^\bullet} Y^\bullet$, respectively, are identified in $\mathcal{D}(\Cb)$ if the exists a complex $Z''^\bullet$ in $\mathcal{D}(\Cb)$ and morphisms $Z^\bullet\xleftarrow{q''^\bullet} Z''^\bullet\xrightarrow{h''^\bullet} Z'^\bullet$ with $q''^\bullet$ a quasi-isomorphism and $h''^\bullet: Z''^\bullet \to Z'^\bullet$ in $\mathcal{K}(\Cb)$ such that $h^\bullet \circ q''^\bullet = h'^\bullet \circ h''^\bullet$ and $q^\bullet \circ q''^\bullet = q'^\bullet \circ h''^\bullet$. If $f^\bullet: X^\bullet \to Y^\bullet$ is a morphism in $\mathcal{D}(\Cb)$ represented by $X^\bullet\xleftarrow{q^\bullet} Z^\bullet\xrightarrow{h^\bullet} Y^\bullet$, then for all $\ell\in \Z$, we let $f^\bullet[\ell]: X^\bullet[\ell]\to Y^\bullet[\ell]$ be represented by $X[\ell]^\bullet\xleftarrow{q^\bullet[\ell]} Z[\ell]^\bullet\xrightarrow{h^\bullet[\ell]} Y^\bullet[\ell]$. Since being a quasi-isomorphism is invariant under shifting, we obtain that $f^\bullet[\ell]$ is also in $\mathcal{D}(\Cb)$. It follows from e.g. \cite[Cor. 10.4.3]{weibel} that $\mathcal{D}(\Cb)$ is also a triangulated category with suspension functor $\Sigma = [1]$ such that if $f^\bullet: X^\bullet \to Y^\bullet$ is a quasi-isomorphism in $\mathcal{C}(\Cb)$ is an isomorphism in $\mathcal{D}(\Cb)$. Moreover, there is a canonical functor $\Cb\to \mathcal{D}(\Cb)$ that sends every object $X$ in $\Cb$ to the complex $\mathcal{D}(\Cb)$ concentrated in degree zero whose non-zero term is $X$, and which we denote by $\overline{X}$. 

\begin{lemma}{(\cite[Lemma 4.1.12]{krause4})}\label{remA1}
Assume that $\Cb$ is an exact category. If $0\to X\to Y\to Z\to 0$ is an admissible exact sequence in $\Cb$, then this sequence induces an distinguished triangle $\overline{X}\to \overline{Y}\to \overline{Z}\to \overline{X}[1]$ in $\mathcal{D}(\Cb)$.
\end{lemma}

\subsection{Cohomology groups and bounded derived categories}\label{remA2}{(\cite[\S 4.1]{krause4})}
Assume next that $\Cb$ is abelian. For all $i\in \Z$,  the {\it $i$-th cohomology group} $\H^i(X^\bullet)$ of $X^\bullet$ is defined as
\begin{equation*}
\H^i(X^\bullet)= \ker d^i/\mathrm{im}\,d^{i-1}. 
\end{equation*}
If $f^\bullet: X^\bullet \to Y^\bullet$ is a morphism in $\mathcal{C}(\Ab)$, then for all $i\in \Z$ there is an induced morphism $\H^i(f^\bullet) : \H^i(X^\bullet)\to \H^i(Y^\bullet)$. If $f^\bullet, g^\bullet: X^\bullet\to Y^\bullet$ are homotopic, then for all $i\in \Z$, $\H^i(f^\bullet) = \H^i(g^\bullet)$. Thus we obtain a well-defined functor $\H^0: \mathcal{K}(\Cb)\to \Cb$ such that $\H^0(\overline{X})=X$ for all $X\in \Ob(\Cb)$. Moreover, $\H^0$ is a homological functor as in (\ref{cohfun}) by e.g. \cite[Prop. 4.1.3]{krause4}.  Thus by \cite[Lemma 4.1.4]{krause4}, $f^\bullet: X^\bullet \to Y^\bullet$ is a quasi-isomorphism if and only if for all $i\in \Z$, $\H^i(f^\bullet): \H^i(X^\bullet)\to \H^i(Y^\bullet)$ is an isomorphism in $\Cb$.
We denote by $\mathcal{D}^b(\Cb)$  the {\it bounded derived category} of $\Cb$ whose objects are those complexes $X^\bullet\in \Ob(\mathcal{D}(\Cb))$ such that $\H^i(X^\bullet)=0$ for almost all $i\in \Z$.  

\subsection{Hereditary categories (\cite[\S 4.4]{krause4})}\label{hereditarycat}

Let $\Cb$ be an abelian category.  A complex $X^\bullet$ in $\mathcal{D}(\Cb)$ is said to be {\it quasi-isomorphic to its cohomology} if there is a quasi-isomorphism between $X^\bullet$ and the complex
\begin{equation*}
\cdots\xrightarrow{0}\H^{i-1}(X^\bullet)\xrightarrow{0} \H^i(X^\bullet)\xrightarrow{0}\H^{i+1}(X^\bullet)\xrightarrow{0}\cdots.
\end{equation*}
In particular, there is an isomorphism in $\mathcal{D}(\Cb)$:
\begin{equation*}
X^\bullet \cong \bigoplus_{i\in \Z}\Sigma^{-i}\H^i(X^\bullet)
\end{equation*}
\noindent
On the other hand we say that $\Cb$ is {\it hereditary} if the functor $\Ext_\Cb^2(-,-)$ vanishes. It follows from \cite[Prop. 4.4.15]{krause4} that $\Cb$ is hereditary if and only if every object in $\mathcal{D}(\Cb)$ is quasi-isomorphic to its cohomology.

\bibliographystyle{amsplain}
\bibliography{Exact_weights}   

\end{document}